\documentclass[12pt, reqno]{amsart}
\usepackage{amsmath}
\usepackage{amsfonts}
\usepackage{amssymb}
\usepackage{latexsym}
\usepackage{amsthm}

\newtheorem{theorem}{Theorem}[section]
\newtheorem{lemma}[theorem]{Lemma}

\newtheorem{coro}[theorem]{Corollary}

\newtheorem{proposition}[theorem]{Proposition}

\newcounter{other}            

\newcommand{\Cn}{\mathbb{C}^n}

\newcommand{\Bn}{\mathbb{B}_ n}

\def\a{\alpha}

\newcommand{\eps}{\varepsilon}

\newcommand{\bn}{{{\mathbb B}_n}}

\newcommand{\C}{{\mathbb C}}
\newcommand{\D}{{\mathbb D}}
\newcommand{\cn}{{\mathbb C}^n}
\newcommand{\B}{{\mathbb B}}

\def\p{\varphi}
\numberwithin{equation}{section}

\begin{document}

\title[BMO and Hankel operators]{Weighted BMO and Hankel operators\\ between Bergman spaces}

\author[J. Pau]
{Jordi Pau}
\address{
Jordi Pau\\
Departament de Matem\`atica Aplicada i Analisi,
Universitat de Barcelona,
08007 Barcelona, Catalonia,
Spain}
\email{jordi.pau@ub.edu}

\author[R. Zhao]
{Ruhan Zhao}
\address{
Ruhan Zhao\\
Department of Mathematics,
State University of New York,
Brockport, NY 14420,
USA}
\email{rzhao@brockport.edu}

\author[K. Zhu]
{Kehe Zhu}
\address{
Kehe Zhu\\
Department of Mathematics and Statistics,
State University of New York,
Albany, NY 12222,
USA}
\email{kzhu@math.albany.edu}



\thanks{This work started when the second named author visited the University of Barcelona in 2013. He thanks the support given by the IMUB during his visit.  The first author was
 supported by DGICYT grant MTM$2011$-$27932$-$C02$-$01$
(MCyT/MEC) and the grant 2014SGR289 (Generalitat de Catalunya)}


\begin{abstract}
We introduce a family of weighted BMO spaces in the Bergman metric on the unit ball of $\cn$
and use them to characterize complex functions $f$ such that the big Hankel operators
$H_f$ and $H_{\bar f}$ are both bounded or compact from a weighted Bergman space
into a weighted Lesbegue space with possibly different exponents and different weights.
As a consequence, when the symbol function $f$ is holomorphic, we characterize bounded and
compact Hankel operators $H_{\bar f}$ between weighted Bergman spaces. In particular, this
resolves two questions left open in \cite{J, Wall}.
\end{abstract}

\maketitle

\section{Introduction}

Let $\Bn$ be the open unit ball in $\Cn$ and $dv$ the usual Lebesgue volume measure on $\Bn$, normalized so that the volume of $\Bn$
is one. Given a parameter $\alpha>-1$ we write
$$dv_{\alpha}(z)=c_{\alpha}\,(1-|z|^2)^{\alpha}dv(z),$$
where $c_{\alpha}$ is a positive constant such that $v_{\alpha}(\Bn)=1$.

Denote by $H(\bn)$ the space of holomorphic functions on $\bn$.
For $0<p<\infty$ the weighted Bergman space $A^p_{\alpha}:=A^p_{\alpha}(\Bn)$
consists of functions $f\in H(\bn)$ that are
in the Lebesgue space $L^p_{\alpha}:=L^p(\Bn,dv_{\alpha})$. The corresponding norm is given by
$$\|f\|_{p,\a}=\left(\int_{\Bn}|f(z)|^p\,dv_{\a}(z)\right)^{1/p}.$$

When $p=2$, the space $A^2_{\alpha}$ is a reproducing kernel Hilbert space: for each $z\in \Bn$ there is a function $K^{\alpha}_ z\in
A^2_{\alpha}$ such that $f(z)=\langle f,K^{\alpha}_ z \rangle _{\alpha}$ whenever $f\in A^2_{\alpha}$. Here
$$\langle f,g \rangle_{\alpha}=\int_{\Bn} f \bar{g}\,dv_{\alpha}$$
is the natural inner product in $L^2_{\alpha}$. $K^{\alpha}_ z$ is called the reproducing kernel of the Bergman space $A^2_{\alpha}$.
It is explicitly given by the formula
$$K^{\alpha}_ z(w)=\frac{1}{(1-\langle w,z\rangle )^{n+1+\alpha}},\quad z,w\in \Bn.$$
We also let $k^{\alpha}_ z$ denote the normalized reproducing kernel at $z$. Thus
$$k^\alpha_z(w)=K^\alpha_z(w)/\sqrt{K^\alpha_z(z)}=\frac{(1-|z|^2)^{(n+1+\alpha)/2}}{(1-\langle w,z\rangle)^{n+1+\alpha}}.$$

The orthogonal projection $P_{\alpha}:L^2_{\alpha}\rightarrow A^2_{\alpha}$ is an integral operator given by
$$P_{\alpha} f(z)=\int_{\Bn} \frac{ f(w)\,dv_{\alpha}(w)}{(1-\langle z,w\rangle )^{n+1+\alpha}},\qquad f\in L^2(\Bn,dv_{\alpha}).$$
The (big) Hankel operator $H_ f^{\beta}$ with symbol $f$ is defined by
$$H^{\beta}_f g=(I-P_{\beta})(fg).$$
We are interested in the mapping properties of $H^\beta_f$ between different Lebesgue spaces.

Hankel operators are closely related to Toeplitz operators and have been extensively studied by many authors in recent decades.
For analytic $f$, Axler \cite{Ax} first characterized the boundedness and compactness of $H_{\bar f}$ on the unweighted Bergman space
of the unit disk. Later on, Axler's result was generalized in \cite{AFP, AJFP} to weighted Bergman spaces of the unit ball in $\mathbb{C}^n$.
For general symbol functions, Zhu \cite{Zhu-VMO} first established the connection between size estimates of Hankel operators and
the mean oscillation of the symbols in the Bergman metric. This idea was further investigated in a series of papers \cite{BCZ0}, \cite{BCZ},
and \cite{BBCZ} in the context of bounded symmetric domains, and in \cite{Li1, Li2} in the context of strongly pseudo convex domains.

The main purpose of this paper is to characterize real-valued functions $f\in L^q_{\beta}$ such that $H^\beta_f$ is bounded or compact
from $A^p_{\alpha}$ to $L^q_{\beta}$ with $1<p\le q<\infty$. This is equivalent to characterizing complex-valued functions $f\in L^q_{\beta}$
such that both $H^\beta_f$ and $H^\beta_{\bar f}$ are bounded or compact between the above spaces.
As a consequence, we will characterize holomorphic symbols $f\in A^1_{\beta}$ such that $H^\beta_{\bar f}$ is bounded or compact
from $A^p_{\alpha}$ to $L^q_{\beta}$ with $1<p\le q<\infty$. Our characterizations are based on a family of weighted BMO spaces in
the Bergman metric.

Most previous results of this type are for bounded and compact Hankel operators from $A^p_{\alpha}$ to $L^p_{\alpha}$.
When $f$ is holomorphic,  Janson \cite{J} and Wallst\'{e}n \cite{Wall} characterized bounded and compact Hankel operators
between weighted Bergman spaces (in the Hilbert space case) with different weights on the unit disk and the unit ball, respectively.
Our results generalize theirs and solve two cases left open by them.

In the following, the notation $A\lesssim B$ means
that there is a positive constant $C$ such that $A\leq CB$,
and the notation $A\asymp B$ means that both $A\lesssim B$ and $B\lesssim A$ hold.

\section{Preliminaries and auxiliary results}

In this section we collect some preliminary results that are needed for the proof of the main theorems. We begin with notation for
the rest of the paper. For any two points $z=(z_1,\dots,z_n)$ and $w=(w_1,\dots,w_n)$ in $\Cn$, we write
$$\langle z,w\rangle=z_1\bar w_1+\cdots+z_n\bar w_n,$$
and
$$|z|=\sqrt{\langle z,z\rangle}=\sqrt{|z_1|^2+\cdots+|z_n|^2}.$$
For any $a\in \Bn$ with $a\neq 0$ we denote by $\p_a(z)$ the M\"obius transformation on $\Bn$
that interchanges the points $0$ and $a$. It is known that
$$\p_a(z)=\frac{a-P_a(z)-s_aQ_a(z)}{1-\langle z,a\rangle},\qquad z\in\Bn,$$
where $s_a =1-|a|^2$ , $P_a$ is the orthogonal projection from $\C^n$ onto the one dimensional subspace $[a]$ generated by $a$,
and $Q_a$ is the orthogonal projection from $\C^n$ onto the orthogonal complement of $[a]$. When $a=0$, $\p_a(z)=-z$.
It is known that $\p_a$ satisfies the following properties:
\begin{equation}\label{eq-pa}
\p_a\circ\p_a(z)=z,\qquad 1-|\p_a(z)|^2=\frac{(1-|a|^2)(1-|z|^2)}{|1-\langle z,a\rangle|^2}.
\end{equation}

For $z,w\in\Bn$, the distance between $z$ and $w$ induced by the Bergman metric is given by
$$\beta(z,w)=\frac12 \,\log\frac{1+|\p_z(w)|}{1-|\p_z(w)|}.$$
For $z\in\Bn$ and $r>0$, the Bergman metric ball at $z$ is given by
$$D(z,r)=\big \{w\in\Bn:\,\beta(z,w)<r \big \}.$$
We refer to \cite{ZhuBn} for more information about automorphisms and the Bergman metric on $\Bn$.

A sequence $\{a_k\}$ of points in $\Bn$ is called a \emph{separated sequence} (in the Bergman metric) if there exists a positive
constant $\delta>0$ such that $\beta(a_i,a_j)>\delta$ for any $i\neq j$. The following result is Theorem 2.23 in \cite{ZhuBn}.

\begin{lemma}\label{covering}
There exists a positive integer $N$ such that for any $0<r<1$ we can find a sequence $\{a_k\}$ in $\Bn$ with the following properties:
\begin{itemize}
\item[(i)] $\Bn=\cup_{k}D(a_k,r)$.
\item[(ii)] The sets $D(a_k,r/4)$ are mutually disjoint.
\item[(iii)] Each point $z\in\Bn$ belongs to at most $N$ of the sets $D(a_k,4r)$.
\end{itemize}
\end{lemma}

Any sequence $\{a_k\}$ satisfying the conditions of the above lemma will be called an $r$-\emph{lattice}
in the Bergman metric. Obviously any $r$-lattice is separated. The following integral estimate is well known and can
be found in \cite[Theorem 1.12]{ZhuBn} for example.

\begin{lemma}\label{Ict}
Let $t>-1$ and $s>0$. There is a positive constant $C$ such that
$$\int_{\Bn} \frac{(1-|w|^2)^t\,dv(w)}{|1-\langle z,w\rangle |^{n+1+t+s}}\le C\,(1-|z|^2)^{-s}$$
for all $z\in \Bn$.
\end{lemma}

We also need a well-known variant of the previous lemma.

\begin{lemma}\label{l2}
Let $\{z_ k\}$ be a separated sequence in $\Bn$ and let $n<t<s$. Then
$$\sum_{k=1}^{\infty}\frac{(1-|z_ k|^2)^t}{|1-\langle z,z_ k \rangle |^s}\le C\,(1-|z|^2)^{t-s},\qquad z\in \Bn.$$
\end{lemma}

Lemma~\ref{l2} above can be deduced from Lemma \ref{Ict} after noticing that, if a sequence $\{z_ k\}$ is separated, then there is a
constant $r>0$ such that the Bergman metric balls $D(z_ k,r)$ are pairwise disjoint. The following result is from \cite{Zhao-Schur}.

\begin{lemma}\label{Z-T}
Given real numbers $b$ and $c$, consider the integral operator on $\Bn$ defined by
$$S_{b,c} f(z)=\int_{\Bn} \frac {f(w)(1-|w|^2)^b\,dv(w)}{|1-\langle z,w\rangle |^c}.$$
Let $1< p\le q<\infty$, $\alpha>-1$, $\beta>-1$, and
$$\lambda=\frac{n+1+\beta}{q}-\frac{n+1+\alpha}{p}.$$
Then the operator $S_{b,c}$ is bounded from $L^p_{\alpha}$ to $L^q_{\beta}$ if and only if
$$\alpha+1<p(b+1),\qquad c\le n+1+b+\lambda.$$
\end{lemma}

We show that, under the same conditions, with an extra (unbounded) factor $\beta(z,w)$ in the integrand, the modified operator
is still bounded from $L^p_{\alpha}$ to $L^q_{\beta}$. Thus we consider the operator
$$T_{b,c}f(z)=\int_{\bn}\frac{f(w)\,\beta(z,w)}{|1-\langle z,w\rangle|^c}(1-|w|^2)^b\,dv(w).$$

\begin{proposition}\label{proj-ge1-beta}
Let $b$ and $c$ be real numbers. Let $1< p\le q<\infty$, $\alpha>-1$, $\beta>-1$, and
$$\lambda=\frac{n+1+\beta}{q}-\frac{n+1+\alpha}{p}.$$
If  $\alpha+1<p(b+1)$ and  $c\le n+1+b+\lambda$, then $T_{b,c}$ is bounded from $L^p_{\alpha}$ to $L^q_{\beta}$.
\end{proposition}

\begin{proof}
Pick $\varepsilon>0$ so that $\alpha+1<p(b+1-\varepsilon)$ and $\beta-q\varepsilon>-1$. Since $\beta(z,w)$ grows logarithmically, we have
$$\beta(z,w)=\beta(0,\varphi_ z(w))\le C (1-|\varphi_ z(w)|^2)^{-\varepsilon}.$$
It follows from \eqref{eq-pa} that
$$|T_{b,c} f(z)| \le C (1-|z|^2)^{-\varepsilon}\int_{\Bn} \frac{(1-|w|^2)^{b-\varepsilon}}{|1-\langle z,w\rangle|^{c-2\varepsilon}}\,|f(w)|\,dv(w).$$
Thus $T_{b,c}$ is bounded from $L^p_{\alpha}$ to $L^q_{\beta}$ if the operator $S_{b-\varepsilon,c-2\varepsilon}$ is bounded
from $L^p_{\alpha}$ to $L^q_{\beta-\varepsilon q}$. The desired result then follows from the previous lemma.
\end{proof}

In a similar manner, the following version of Lemma \ref{Ict} can be obtained. The proof is left to the interested reader.

\begin{lemma}\label{Ict-beta}
Let $t>-1$, $s>0$, and $d>0$. There is a positive constant $C$ such that
$$\int_{\Bn} \frac{(1-|w|^2)^t\,\beta(z,w)^d\,dv(w)}{|1-\langle z,w\rangle |^{n+1+t+s}}\le C\,(1-|z|^2)^{-s}$$
for all $z\in \Bn$.
\end{lemma}

The rest of this section is devoted to the proof of Theorem \ref{5}, which can be interpreted as some sort of
tangential maximum principle. We begin with the following elementary fact.

\begin{lemma}
Suppose $F$ and $G$ are holomorphic functions on $\B_2$. If
\begin{equation}
F(z_1,z_2)=G(z_1,z_2),\qquad z_k=(u/\sqrt2)e^{i\theta_k},
\label{eq1}
\end{equation}
where $\theta_k$ are arbitrary real numbers and $u$ is arbitrary from the unit disk $\D$, then $F=G$ on $\B_2$.
\label{4}
\end{lemma}

\begin{proof}
Suppose
$$F(z_1,z_2)=\sum_{k,l=0}^\infty a_{kl}z_1^kz_2^l,
\qquad G(z_1,z_2)=\sum_{k,l=0}^\infty b_{kl}z_1^kz_2^l.$$
Then we have
$$\sum_{k,l=0}^\infty a_{kl}\left(\frac u{\sqrt2}\right)^{k+l}e^{ik\theta_1}e^{il\theta_2}
=\sum_{k,l=0}^\infty b_{kl}\left(\frac u{\sqrt2}\right)^{k+l}e^{ik\theta_1}e^{il\theta_2}$$
for all $u\in\D$ and all real $\theta_1$ and $\theta_2$. By the uniqueness of Fourier coefficients
on the torus, we must have $a_{kl}=b_{kl}$ for all $k$ and $l$. This shows that $F=G$ on $\B_2$.
\end{proof}

For $n>1$, $f\in H(\bn)$, and $z\in\bn-\{0\}$ we will write
$$|\nabla_tf(z)|=\sup\left\{\left|\frac{\partial f}{\partial u}(z)\right|:\|u\|=1,u\in[z]^\perp\right\}$$
and call it the complex tangential gradient of $f$ at $z$.

\begin{theorem}\label{5}
Let $n>1$ and  $f\in H(\bn)$. If $|\nabla_tf(z)|\to0$ as $|z|\to1^-$,  then $f$ is constant.
\end{theorem}

\begin{proof}
We will first prove the case  $n=2$. In this case, the condition
\begin{equation}
|\nabla_tf(z)|\to0,\qquad |z|\to1^-,
\label{eq2}
\end{equation}
is equivalent to
$$\lim_{|z|\to1^-}\left(\overline z_2\frac{\partial f}{\partial z_1}(z)-
\overline z_1\frac{\partial f}{\partial z_2}(z)\right)=0.$$
Let
$$z_1=\frac{u}{\sqrt2}e^{i\theta_1},\qquad z_2=\frac u{\sqrt2}e^{i\theta_2},$$
where $u\in\D$ and $\theta_k$ are arbitrary real numbers. Since
$$|z_1|^2+|z_2|^2=|u|^2,$$
we see that $|z|\to1^-$ if and only if $|u|\to1^-$. Thus condition (\ref{eq2}) implies that
$$\lim_{|u|\to1^-}\left[e^{-i\theta_2}\frac{\partial f}{\partial z_1}\left(\frac u{\sqrt2}
e^{i\theta_1},\frac u{\sqrt2}e^{i\theta_2}\right)-e^{-i\theta_1}\frac{\partial f}{\partial z_2}
\left(\frac u{\sqrt2}e^{i\theta_1},\frac u{\sqrt2}e^{i\theta_2}\right)\right]=0.$$
For any fixed $\theta_k$ the expression inside the brackets above is an analytic function $F(u)$ on
the unit disk. By the classical maximum principle for analytic functions on the unit disk, we conclude that the
condition in (\ref{eq2}) implies that $F(u)$ is identically zero on $\D$. Therefore, the condition in (\ref{eq2}) implies that
$$e^{i\theta_1}\frac{\partial f}{\partial z_1}\left(\frac u{\sqrt2}e^{i\theta_1},
\frac u{\sqrt2}e^{i\theta_2}\right)=e^{i\theta_2}\frac{\partial f}{\partial z_2}
\left(\frac u{\sqrt2}e^{i\theta_1},\frac u{\sqrt2}e^{i\theta_2}\right)$$
for all $u\in\D$. Multiply the above equation by $u/\sqrt2$, we conclude that
\begin{equation}
z_1\frac{\partial f}{\partial z_1}(z_1,z_2)=z_2\frac{\partial f}{\partial z_2}(z_1,z_2)
\label{eq3}
\end{equation}
whenever
$$z_k=\frac u{\sqrt2}e^{i\theta_k},\qquad k=1,2.$$
By Lemma \ref{4}, we see that the condition in (\ref{eq2}) implies
that the identity in (\ref{eq3}) must hold for all $z=(z_1,z_2)$ in the unit ball.

Write
$$f(z)=\sum_{i,j=0}^\infty a_{ij}z_1^iz_2^j$$
and assume that the identity in (\ref{eq3}) holds for all $z=(z_1,z_2)\in\B_2$. Then
$$\sum_{i,j=0}^\infty ia_{ij}z_1^iz_2^j=\sum_{i,j=0}^\infty ja_{ij}z_1^iz_2^j.$$
This gives $ia_{ij}=ja_{ij}$ for all $i$ and $j$, which implies that $a_{ij}=0$ whenever $i\not=j$. Writing $a_j=a_{jj}$, we obtain
$$f(z)=\sum_{j=0}^\infty a_j(z_1z_2)^j.$$
In this case, we have
$$\overline z_2\frac{\partial f}{\partial z_1}(z)-\overline z_1
\frac{\partial f}{\partial z_2}(z)=(|z_2|^2-|z_1|^2)\sum_{j=1}^\infty ja_j(z_1z_2)^{j-1}.$$
Consider the case in which
$$z_1=u\sin\theta,\qquad z_2=u\cos\theta,$$
where $u\in\D$ is arbitrary and $0<\theta<\pi/4$ is fixed. Then
$$\overline z_2\frac{\partial f}{\partial z_1}(z)-\overline z_1\frac{\partial f}{\partial z_2}(z)
=|u|^2\cos(2\theta)\sum_{j=1}^\infty ja_j(u^2\sin\theta\cos\theta)^{j-1}.$$
Let $|u|\to1^-$ and apply the classical maximum principle on the unit disk, we must have $a_j=0$ for all $j\ge1$,
namely, $f$ is constant. This completes the proof of the theorem in the case $n=2$.

Next let us assume that $n\ge3$ and $|\nabla_tf(z)|\to0$ as $|z|\to1^-$ for some $f\in H(\bn)$. We want to show that $f$ is
constant. For any $z=(z_1,z_2,z_3,\cdots,z_n)\in\bn$, the vector $(z_2,-z_1,0,\cdots,0)$ is perpendicular to $z$. Therefore,
$$\lim_{|z|\to1^-}\left(\overline z_2\frac{\partial f}{\partial z_1}(z)-\overline z_1
\frac{\partial f}{\partial z_2}(z)\right)=0.$$
We proceed to show that this implies that $f$ is independent of the first two variables.

Fix $z=(z_1, z_2, z_3,\cdots,z_n)\in\bn$ (we specifically mention that it is OK if some of the $z_k$ are $0$) and write
$$1-r^2=|z_3|^2+\cdots+|z_n|^2,$$
where $0<r\le1$. Consider the function
$$g(w_1,w_2)=f(rw_1,rw_2,z_3,\cdots,z_n),$$
where $w=(w_1,w_2)\in\B_2$. It is clear that $|w|\to1^-$ if and only if
$$|rw_1|^2+|rw_2|^2+|z_3|^2+\cdots+|z_n|^2\to1^-.$$
So we also have
$$\lim_{|w|\to1^-}\left[\overline w_2\frac{\partial g}{\partial w_1}(w)-\overline w_1
\frac{\partial g}{\partial w_2}(w)\right]=0.$$
By the $n=2$ case that we have already proved, $g$ must be constant. If we choose
$w\in\B_2$ such that $rw_k=z_k$ for $k=1,2$. Then
$$f(z_1,z_2,z_3,\cdots,z_n)=g(w_1,w_2)=g(0,0)=f(0,0,z_3,\cdots,z_n).$$

Repeat the argument for the first and $k$th variable, where $k\ge3$, and let $k$ run from
$3$ to $n$. The result is
$$f(z_1,z_2,z_3,\cdots,z_n)=f(0,0,0,\cdots,0).$$
Since $z=(z_1,\cdots,z_n)\in\bn$ is arbitrary, we have shown that $f$ is constant.
\end{proof}

For $f\in H(\bn)$ we write
$$\nabla f(z)=\left(\frac{\partial f}{\partial z_1}(z),\cdots,\frac{\partial f}{\partial z_n}(z)\right),\qquad z\in \Bn$$
and call $|\nabla f(z)|$ the complex gradient of $f$ at $z$. As a consequence of Theorem \ref{5}, we obtain the following maximum
principle in terms of the invariant gradient $\widetilde{\nabla} f(z)=\nabla (f\circ \varphi_ z)(0)$.

\begin{coro}\label{C5}
Let $n>1$ and  $f\in H(\bn)$. If $(1-|z|^2)^{-1/2}|\widetilde\nabla f(z)|\to0$ as $|z|\to1^-$, then $f$ is constant.
\end{coro}

\begin{proof}
This follows from \cite[Theorem 7.22]{ZhuBn} and Theorem \ref{5}.
\end{proof}

\section{A family of weighted BMO spaces}

Let $\gamma \in \mathbb{R}$. For any positive radius $r$ and every exponent $p$ with $1\le p<\infty$, the space $BMO^p_{r,\gamma}$
consists of those functions $f\in L^p_{loc}(\Bn)$ (the space of locally $L^p$ integrable functions  on $\Bn$) such that
$$\|f\|_{BMO^p_{r,\gamma}}=\sup \big \{ (1-|z|^2)^{\gamma} MO_{p,r}(f)(z):z\in \Bn \big \}<\infty,$$
where
$$MO_{p,r}(f)(z)= \left [\frac{1}{v_{\sigma}(D(z,r))}\int_{D(z,r)} |f(w)-\widehat{f_ r}(z)|^p\,dv_{\sigma}(w) \right ]^{1/p}$$
is the $p$-mean oscillation of $f$ at $z$ in the Bergman metric. Here
$$\widehat{f_ r}(z)=\frac{1}{v_{\sigma}(D(z,r))} \int_{D(z,r)} f(w)\,dv_{\sigma}(w)$$
is the averaging function of $f$ and $dv_\sigma(z)=(1-|z|^2)^\sigma\,dv(z)$. At first glance, the function $MO_{p,r}(f)$ seems to depend
on the real parameter $\sigma$, but the weight factor $(1-|z|^2)^{\sigma}$ in $dv_{\sigma}$ is essentially canceled out by the extra factor
$(1-|z|^2)^{\sigma}$ in $v_{\sigma}(D(z,r))\asymp (1-|z|^2)^{n+1+\sigma}$. As a consequence, the space $BMO^p_{r,\gamma}$ is actually
independent of the weight parameter $\sigma$. In particular, this independence on $\sigma$ is a consequence of the following lemma.

\begin{lemma}\label{Lem-1}
Let $1\le p<\infty$, $f\in L^p_{loc}(\Bn)$, and $r>0$. Then $f\in BMO^p_{r,\gamma}$ if and only if there exists some constant $C>0$ such
that for any $z\in \Bn$, there is a constant $\lambda_ z$ satisfying
\begin{equation}\label{Eq-L1}
\frac{(1-|z|^2)^{\gamma p}}{v_{\sigma}(D(z,r))}\int_{D(z,r)} |f(w)-\lambda_ z|^p\,dv_{\sigma}(w) \le C.
\end{equation}
\end{lemma}

\begin{proof}
If $f\in BMO^p_{r,\gamma}$, then \eqref{Eq-L1} holds with $C=\|f\|_{BMO^p_{r,\gamma}}$ and $\lambda_ z=\widehat{f_ r}(z)$. Conversely,
if \eqref{Eq-L1} is satisfied, then by the triangle inequality for the $L^p$-norm, $MO_{p,r}(f)(z)$ is less than or equal to
$$\left[\frac{1}{v_{\sigma}(D(z,r))}\int_{D(z,r)}|f(w)-\lambda_ z|^p\,dv_{\sigma}(w)\right ]^{\frac1p}
+|\widehat{f_ r}(z)-\lambda_ z|.$$
By H\"{o}lder's inequality,
\begin{eqnarray*}
|\widehat{f_ r}(z)-\lambda_ z|&=&\left | \frac{1}{v_{\sigma}(D(z,r))}\int_{D(z,r)}\!\! (f(w)-\lambda_ z)\,dv_{\sigma}(w)\right |\\
&\le& \left[\frac{1}{v_{\sigma}(D(z,r))}\int_{D(z,r)}\!\! |f(w)-\lambda_ z|^p\,dv_{\sigma}(w)\right ]^{\frac1p}.
\end{eqnarray*}
It follows that
\[ (1-|z|^2) ^{\gamma} MO_{p,r}(f)(z) \le 2 C^{1/p}\]
for all $z\in \Bn$, so that $f\in BMO^p_{r,\gamma}$.
\end{proof}

For a continuous function $f$ on $\Bn$ let
$$\omega_ r(f)(z)=\sup\big \{ |f(z)-f(w)|:w\in D(z,r)\big \}.$$
The function $\omega_ r(f)(z)$ is called the oscillation of $f$ at the point $z$ in the Bergman metric. For any $r>0$ and $\gamma \in
\mathbb{R}$, let $BO_{r,\gamma}$ denote the space of continuous functions $f$ on $\Bn$ such that
$$\|f\|_{BO_{r,\gamma}}=\sup_ {z\in\Bn} (1-|z|^2)^{\gamma} \omega_ r(f)(z)<\infty.$$

\begin{lemma}\label{L-CBO}
Let $r>0$,  $\gamma \in \mathbb{R}$, and $f$ be a continuous function on $\Bn$. If $\gamma\ge0$, then
$f\in BO_{r,\gamma}$ if and only if there is a constant $C>0$ such that
$$|f(z)-f(w)|\le C \frac{\beta(z,w)+1}{\min (1-|z|,1-|w|)^{\gamma}}$$
for all $z$ and $w$ in $\bn$. If $\gamma<0$, then $f\in BO_{r,\gamma}$ if and only if there is a constant $C>0$ such that
$$|f(z)-f(w)|\le C \frac{\beta(z,w)+1}{\big [\max (1-|z|,1-|w|)\big ]^{-\gamma}}\,|1-\langle z,w \rangle |^{-2\gamma}$$
for all $z$ and $w$ in $\Bn$.
\end{lemma}

\begin{proof}
Assume that $f\in BO_{r,\gamma}$. If $\beta(z,w)\le r$, the result is clear, because then
$$|1-\langle z,w \rangle | \asymp 1-|z| \asymp 1-|w|.$$
Fix any $z,w\in \Bn$ with $\beta(z,w)>r$. Let $\lambda(t)$, $0\le t\le 1$, be the geodesic in the Bergman metric from $z$ to $w$.
Let $N=[\beta(z,w)/r]+1$ and $t_ i=i/N$, $0\le i\le N$, where $[x]$ denotes the largest integer less than or equal to $x$.
Set $z_i=\lambda (t_ i)$, $0\le i\le N$. Then
$$\beta(z_ i,z_{i+1})=\frac{\beta(z,w)}{N}\le r.$$
Therefore,
\begin{eqnarray*}
|f(z)-f(w)|&\le& \sum_{i=1}^{N} |f(z_{i-1})-f(z_ i)| \le \sum_{i=1}^{N} \omega_ r(f)(z_ i)\\
& \le &\|f\|_{BO_{r,\gamma}} \sum_{i=1}^{N} (1-|z_ i|)^{-\gamma} .
\end{eqnarray*}

If $\gamma \ge 0$, it follows from the obvious inequality
$$(1-|z_ i|)\ge \min(1-|z|,1-|w|)$$
that
\begin{eqnarray*}
|f(z)-f(w)| &\le& \|f\|_{BO_{r,\gamma}} \frac{N} {\min (1-|z|,1-|w|)^{\gamma}}\\
&\le& \|f\|_{BO_{r,\gamma}}\,\frac{\beta(z,w)/r+1}{\min (1-|z|,1-|w|)^{\gamma}}\\
&\le& \max(1,1/r)\, \|f\|_{BO_{r,\gamma}}\, \frac{\beta(z,w)+1}{\min (1-|z|,1-|w|)^{\gamma}}.
\end{eqnarray*}

If $\gamma<0$, the result is proved in the same way once the inequality
\begin{equation}\label{eq-geo}
1-|z_ i| \le \frac{2\,|1-\langle z,w \rangle |^2}{\max(1-|z|^2,1-|w|^2)}
\end{equation}
is established. To prove this, simply note that the M\"{o}bius transformation $\varphi_ z$ sends the geodesic joining $z$ and $w$ to
the geodesic joining $0$ and $\varphi_ z(w)$. This gives
$$1-|\varphi_ z(w)|^2 \le 1-|\varphi_ z(z_ i)|^2.$$
Developing this inequality using \eqref{eq-pa}, we get
$$\frac{1-|w|^2}{|1-\langle z,w \rangle |^2} \le \frac{1-|z_ i|^2}{|1-\langle z,z_ i  \rangle |^2}\le \frac{2}{|1-\langle z,z_ i \rangle |},$$
which gives
$$|1-\langle z,z_ i  \rangle | \le \frac{2\, |1-\langle z,w \rangle |^2}{1-|w|^2}.$$
Interchanging the roles of $z$ and $w$, we get \eqref{eq-geo}.

The converse implication is obvious.
\end{proof}

A consequence of the above lemma is that the space $BO_{r,\gamma}$ is independent of the choice of $r$. So we will simply
write $BO_{\gamma}=BO_{1,\gamma}$ and
$$\|f\|_{BO_{\gamma}}=\|f\|_{BO_{1,\gamma}}=\sup_ {z\in\Bn} (1-|z|^2)^{\gamma} \omega_ 1(f)(z).$$

Let $0<p<\infty$, $\gamma \in \mathbb{R}$, and $r>0$. We say that $f\in BA^p_{r,\gamma}$ if $f\in L^p_{loc}(\Bn)$ and
$$\|f\|_{BA^p_{r,\gamma}}=\sup_{z\in \Bn} (1-|z|^2)^{\gamma} \left [ \widehat{|f|^p_ r}(z)\right ]^{1/p}<\infty.$$
We proceed to show that the space $BA^p_{r,\gamma}$ is also independent of $r$.

For $\sigma>-1$ and $c>0$ the generalized Berezin transform $B_{c,\sigma}(\varphi)$ of a function $\varphi\in L^1(\Bn,dv_{\sigma})$
is defined as
$$B_{c,\sigma} (\varphi) (z)= (1-|z|^2)^{c}\int_{\Bn} \frac{\varphi (w)}{|1-\langle w,z\rangle|^{n+1+c+\sigma}} \,dv_ {\sigma}(w).$$
In the case when $c=n+1+\sigma$, this coincides with the ordinary Berezin transform
$B_{\sigma} \varphi (z)= \langle \varphi k^{\sigma}_ z,k^{\sigma}_ z\rangle_{\sigma}$.

\begin{lemma}\label{CMf}
Let $0<p\le q<\infty$, $\alpha,\beta>-1$, and $f\in L^q_{loc}(\Bn)$. Set
$$\gamma=(n+1+\beta)/q-(n+1+\alpha)/p,\qquad d\mu_{f,\beta}=|f|^q dv_{\beta}.$$
The following conditions are equivalent:
\begin{itemize}
\item[(i)] The embedding $i:A^p_{\alpha}\rightarrow L^q(\Bn,d\mu_{f,\beta})$ is bounded.
\item [(ii)] $f\in BA^q_{r,\gamma}$ for some (or all) $r>0$.
\item[(iii)] $(1-|z|^2)^{\gamma q} B_{c,\sigma}(|f|^q)\in L^{\infty}(\Bn)$ for all $\sigma>-1+\gamma q$ and all $c>\max(0,-\gamma q)$.
\end{itemize}
\end{lemma}

\begin{proof}
By \cite[Theorem 50]{ZZ}, condition (i) is equivalent to
$$\mu_ {f,\beta} (D(z,r))\le C (1-|z|^2)^{(n+1+\alpha)q/p}$$
for some (or all) $r>0$. Since
$$\widehat{|f|^q_ r}(z)\asymp\frac{\mu_ {f,\beta} (D(z,r))}{(1-|z|^2)^{n+1+\beta}},$$
it follows that (i) and (ii) are equivalent.

Since $|1-\langle z,w \rangle |\asymp (1-|z|^2)$ for $w\in D(z,r)$, we have
\begin{eqnarray*}
B_{c,\sigma}(|f|^q)(z)&=&(1-|z|^2)^{c}\int_{\Bn}\frac{|f(w)|^q\,dv_{\sigma}(w)}{|1-\langle z,w \rangle |^{n+1+c+\sigma}}\\
&\ge& (1-|z|^2)^{c}\int_{D(z,r)}\frac{|f(w)|^q\,dv_{\sigma}(w)}{|1-\langle z,w \rangle |^{n+1+c+\sigma}}\\
&\asymp& \widehat{|f|^q_ r}(z).
\end{eqnarray*}
This proves that (iii) implies (ii).

To finish the proof, let $\{a_j\}$ be an $r$-lattice. Then
\begin{eqnarray*}
B_{c,\sigma}(|f|^q)(z)&=&(1-|z|^2)^{c}\int_{\Bn}\frac{|f(w)|^q\,dv_{\sigma}(w)}{|1-\langle z,w \rangle |^{n+1+c+\sigma}}\\
& \le& (1-|z|^2)^{c}\sum_{j=1}^\infty\int_{D(a_ j,r)}\frac{|f(w)|^q\,dv_{\sigma}(w)}{|1-\langle z,w \rangle |^{n+1+c+\sigma}}.
\end{eqnarray*}
By the estimate in (2.20) on page 63 of \cite{ZhuBn}, we have
$$|1-\langle z,w \rangle |\asymp |1-\langle z,a_ j \rangle |,\qquad w\in D(a_ j,r).$$
Thus
$$B_{c,\sigma}(|f|^q)(z) \lesssim (1-|z|^2)^{c}\sum_{j=1}^\infty\frac{(1-|a_ j|^2)^{\sigma-\beta}}{|1-\langle z,a_ j
\rangle |^{n+1+c+\sigma}}\,\,\mu_ {f,\beta} (D(a_ j,r)).$$
So condition (ii) implies that
$$B_{c,\sigma}(|f|^q)(z) \lesssim  (1-|z|^2)^{c}\sum_{j=1}^\infty\frac{(1-|a_ j|^2)^{n+1+\sigma-\gamma q}}{|1-\langle z,a_ j
\rangle |^{n+1+c+\sigma}}.$$
Since $\sigma>-1+\gamma q$, or $n+1+\sigma-\gamma q>n$, an application of Lemma \ref{l2} shows that
condition (ii) implies (iii).
\end{proof}


As a consequence of the previous result, we see that the space $BA^p_{r,\gamma}$ is independent of the choice of $r$.
Thus we will simply call it $BA^p_{\gamma}$.

\begin{lemma}\label{CBO}
Let $f\in L^1_{loc}(\Bn)$ and $r>0$. Then $\widehat{f}_ r$ is continuous.
\end{lemma}

\begin{proof}
The proof is elementary and we omit the details here.
\end{proof}

\begin{theorem}\label{BMO-S}
Suppose $r>0$, $\gamma \in \mathbb{R}$, $1\le p<\infty$, and $f\in L^p_{loc}(\Bn)$. The following conditions are equivalent:
\begin{itemize}
\item[(a)] $f\in BMO^p_{r,\gamma}$.
\item[(b)] $f=f_ 1+f_ 2$ with $f_ 1\in BO_{\gamma}$ and $f_ 2 \in BA^p_{\gamma}$.
\item[(c)] For some (or all) $\sigma>\max(-1,-1+\gamma p)$ and for each $c>\max(0,-2\gamma p)$ we have
$$\sup_{z\in \Bn}  \int_{\Bn} |f(w)-\widehat{f_ r}(z)|^p \frac{(1-|z|^2)^{c+\gamma p}}{|1-\langle z,w \rangle |^{n+1+c+\sigma}}
\,dv_{\sigma}(w)<\infty.$$
\item[(d)] For some (or all) $\sigma>\max(-1,-1+\gamma p)$ and for each $c>\max(0,-2\gamma p)$ there is a function
$\lambda_ z$ such that
$$\sup_{z\in \Bn}  \int_{\Bn} |f(w)-\lambda_ z|^p \frac{(1-|z|^2)^{c+\gamma p}}{|1-\langle z,w \rangle |^{n+1+c+\sigma}}
\,dv_{\sigma}(w)<\infty.$$
\end{itemize}
\end{theorem}

\begin{proof}
That (c) implies (d) is obvious, and the implication (d) $\Rightarrow$ (a) is a consequence of Lemma \ref{Lem-1} and the inequality
\begin{eqnarray*}
&&\frac{(1-|z|^2)^{\gamma p}}{v_{\sigma}(D(z,r))}\int_{D(z,r)} \!\! |f(w)-\lambda_ z|^p\,dv_{\sigma}(w)\\
&\lesssim&\int_{\Bn} \! |f(w)-\lambda_ z|^p\frac{(1-|z|^2)^{c+\gamma p}}{|1-\langle z,w \rangle |^{n+1+c+\sigma}} dv_{\sigma}(w),
\end{eqnarray*}
which follows from the well-known facts that
$$|1-\langle z,w\rangle |\asymp (1-|z|^2),\qquad v_{\sigma}(D(z,r))\asymp (1-|z|^2)^{n+1+\sigma},$$
for all $z\in\Bn$ and $w\in D(z,r)$.

The proof of (a) $\Rightarrow$ (b) can be done as in \cite[Theorem 5]{Zhu-Pac}. Indeed, since $r$ is arbitrary, it suffices to show that
$$BMO^p_{2r,\gamma}\subset BO_{\gamma}+ BA^p_{\gamma}.$$
Given $f\in BMO^p_{2r,\gamma}$ and points $z,w\in \Bn$ with $\beta(z,w)\le r$, we have
\begin{eqnarray*}
 |\widehat{f_ r}(z)-\widehat{f_ r}(w)| &\le& |\widehat{f_ r}(z)-\widehat{f}_ {2r}(z)|+|\widehat{f}_{ 2r}(z)-\widehat{f_ r}(w)|\\
 & \le& \frac{1}{v_{\sigma}(D(z,r))}\int_{D(z,r)} \!\!|f(u)-\widehat{f}_ {2r}(z)|\,dv_{\sigma}(u)\\
 &&\quad +\frac{1}{v_{\sigma}(D(w,r))}\int_{D(w,r)}\!\! |f(u)-\widehat{f}_ {2r}(z)|\,dv_{\sigma}(u).
\end{eqnarray*}
Since $v_{\sigma}(D(w,r))\asymp v_{\sigma}(D(z,r))$ for $w\in D(z,r)$, and since $D(z,r)$ and $D(w,r)$ are both contained in
$D(z,2r)$, it follows from H\"{o}lder's inequality that the two integral summands above are both bounded by constant times
$(1-|z|^2)^{-\gamma}\|f\|_{BMO^p_{2r,\gamma}}$. This together with Lemma \ref{CBO} proves that $\widehat{f_ r}$ belongs to
$BO_ {r,\gamma}=BO_{\gamma}$. On the other hand, we can prove that the function $g=f-\widehat{f_ r}$ is in $BA^p_{\gamma}$
whenever $f\in BMO^p_{2r,\gamma}$. In fact, it is rather easy to see that $f\in BMO^p_{2r,\gamma}$ implies that
$f\in BMO^p_{r,\gamma}$. By the triangle inequality for $L^p$ integrals,
$$\big[\widehat{|g|^p}_ r (z)\big ]^{1/p}\le\left[\frac{1}{v_{\sigma}(D(z,r))}\int_{D(z,r)} \!\!|f(u)-\widehat{f}_ {r}(z)|^p
\,dv_{\sigma}(u)\right]^{\frac1p}+ \omega_ r(\widehat{f_ r})(z).$$
Since $ \widehat{f_ r} \in BO_{r,\gamma}$ and $f\in BMO^p_{r,\gamma}$, we deduce that $g$ belongs to $BA^p_{\gamma}$.

To show that (b) implies (c), first observe that it follows from Lemma \ref{Ict} that the integral appearing in part (c) is dominated by
$$(1-|z|^2)^{\gamma p} \Big ( B_{c,\sigma}(|f|^p)(z)+ |\widehat{f}_ r(z)|^p\Big ),$$
and by H\"{o}lder's inequality, we have $|\widehat{f}_ r(z)|^p\le \widehat{|f|^p_ r} (z)$. Thus Lemma~\ref{CMf} shows that
$f\in BA^p_\gamma$ implies condition (c). On the other hand, if $f\in BO_{\gamma}$, we write
$$f(w)-\widehat{f_ r}(z) =\frac{1}{v_{\sigma}(D(z,r))} \int_{D(z,r)}\big (f(w)- f(\zeta)\big )\,dv_{\sigma}(\zeta)$$
and use Lemma \ref{L-CBO} and the triangle inequality to obtain
$$|f(w)-\widehat{f_ r}(z)|\le C \|f\|_{BO_{\gamma}} \,\frac{\beta(z,w)+1}{\min (1-|z|,1-|w|)^{\gamma}}, \qquad \gamma\ge 0,$$
and
$$|f(w)-\widehat{f_ r}(z)|\le C \|f\|_{BO_{\gamma}} \,\frac{(\beta(z,w)+1)\,|1-\langle z,w\rangle |^{-2\gamma}}{(1-|z|)^{-\gamma}},
\qquad \gamma < 0.$$
In both cases,  the integral estimates in Lemmas \ref{Ict} and \ref{Ict-beta} show that (c) holds if $f\in BO_{\gamma}$. This
shows that condition (b) implies (c) and completes the proof of the theorem.
\end{proof}

One of the consequences of the above result is that the space $BMO^p_{r,\gamma}$ is also independent of $r$. So from now on
it will simply be denoted by $BMO^p_{\gamma}$.

Next we are going to identify the space of all holomorphic functions in $BMO^p_{\gamma}$ with certain Bloch-type spaces.
Recall that for $\alpha\ge0$, the Bloch type space $\mathcal{B}^{\alpha}=\mathcal{B}^{\alpha}(\Bn)$ consists of holomorphic functions
$f$ in $\Bn$ for which
$$\|f\|_{\mathcal{B}^{\alpha}}=|f(0)|+\sup_{z\in \Bn} (1-|z|^2)^{\alpha} |\nabla f(z)|<\infty.$$
Note that the complex gradient $\nabla f(z)$ can be replaced by the radial derivative $Rf(z)$. We will simply obtain equivalent norms.
When $\alpha>1/2$, a description can also be obtained using the invariant gradient $\widetilde{\nabla} f(z)=
\nabla (f\circ \varphi_ z)(0)$. That is, for $\alpha>1/2$, a holomorphic function $f$ belongs to $\mathcal{B}^{\alpha}$ if and only if
$$|f(0)|+\sup_{z\in \Bn} (1-|z|^2)^{\alpha-1}|\widetilde{\nabla} f(z)|<\infty,$$
and this quantity defines an equivalent norm in $\mathcal{B}^{\alpha}$. Note that when $n=1$, this is true for all $\alpha\ge0$,
because in this case we actually have $\widetilde{\nabla} f(z)=(1-|z|^2) f'(z)$.

When $0<\alpha<1$, the Bloch type space $\mathcal{B}^{\alpha}$ coincides (with equivalent norms) with the holomorphic Lipschitz
space $\Lambda_{1-\alpha}=\Lambda_{1-\alpha}(\Bn)$ consisting of all holomorphic functions $f$ in $\Bn$  such that
$$\|f\|_{\Lambda_{1-\alpha}}=|f(0)|+\sup \left \{ \frac{|f(z)-f(w)|}{|z-w|^{1-\alpha}}:\,z,w\in \Bn,\,z\neq w\right \}<\infty .$$
Note that when $\alpha=0$, the space ${\mathcal B}^\alpha$ consists of holomorphic functions $f$ with bounded partial derivatives.
Equivalently, ${\mathcal B}^0$ consists of all holomorphic functions $f$ such that
$$\sup\left\{\frac{|f(z)-f(w)|}{|z-w|}:z\not=w\right\}<\infty.$$
This space is not what is usually called the Lipschitz space $\Lambda_1$.
We refer to \cite[Chapter 7]{ZhuBn} for all these properties of Bloch and Lipschitz type spaces.

\begin{proposition}\label{Ba}
Let $\gamma \in \mathbb{R}$, $1\le p<\infty$, and $f\in H(\Bn)$. Then
$f\in  BMO^p_{\gamma}$ if and only if
$$\|f\|_{\gamma,*}=\sup_{z\in \Bn} (1-|z|^2)^{\gamma} |\widetilde{\nabla }f(z)|<\infty.$$
\end{proposition}

\begin{proof}
We will show that  condition (d) of Theorem \ref{BMO-S} is satisfied with $\lambda_ z=f(z)$ if $\|f\|_{\gamma,*}<\infty$.
To this end, let $\sigma>\max(-1,-1+p\gamma)$ and $c>\max(0,-\gamma p)$. By \cite[Lemma 7]{PZ-1} and Lemma \ref{Ict}, we have
\begin{eqnarray*}
&&\int_{\Bn} |f(w)-f(z)|^p \frac{(1-|z|^2)^{c+\gamma p}}{|1-\langle z,w \rangle |^{n+1+c+\sigma}} dv_{\sigma}(w)\\
&\le& C \int_{\Bn} |\widetilde{\nabla} f(w)|^p \frac{(1-|z|^2)^{c+\gamma p}}{|1-\langle z,w \rangle |^{n+1+c+\sigma}} dv_{\sigma}(w)\\
&\le& C \|f\|^p_{\gamma,*} \,(1-|z|^2)^{c+\gamma p} \int_{\Bn} \frac{dv_{\sigma-\gamma p}(w)}{|1-\langle z,w \rangle |^{n+1+c+\sigma}}\\
&\le& C.
\end{eqnarray*}
Thus $\|f\|_{\gamma,*}<\infty$ implies that $f\in BMO^p_\gamma$.

To prove the other implication, we use the inequality
$$|\widetilde{\nabla}f(z)|^p \lesssim \frac{1}{v_{\sigma}(D(z,r))} \int_{D(z,r)} |f(w)-f(z)|^p\,dv_{\sigma}(w),$$
which appears on page 182 of \cite{ZhuBn}. By the triangle inequality for $L^p$ spaces,
$$|\widetilde\nabla f(z)|\le MO_{p,r}(f)(z)+|f(z)-\widehat f_r(z)|.$$
Applying Lemma 2.24 of \cite{ZhuBn} to the function $g(w)=f(w)-\widehat f_r(z)$ and the point $z$, we find a constant $C$ such that
$$|f(z)-\widehat f_r(z)|\le C\,MO_{p,r}(f)(z)$$
for all $z\in\Bn$. Thus $f\in BMO^p_\gamma$ implies $\|f\|_{\gamma,*}<\infty$.
\end{proof}

\begin{coro}
Let $\gamma \in \mathbb{R}$ and $1\le p<\infty$. Then
\begin{enumerate}
\item[(a)] If $n=1$, we have
$$H(\Bn)\cap BMO^p_{\gamma}=\begin{cases}
\mathcal{B}^{1+\gamma}, & \gamma \ge -1,\cr {\mathbb C}, & \gamma<-1.\end{cases}$$
\item[(b)] If $n>1$, we have
$$H(\Bn)\cap BMO^p_{\gamma}=\begin{cases}
\mathcal{B}^{1+\gamma}, &\gamma > -1/2,\cr {\mathbb C}, & \gamma<-1/2.\end{cases}$$
\item[(c)] If $n>1$ and $\gamma=-\frac12$, the space $H(\Bn)\cap BMO^p_{\gamma}$ consists of those holomorphic
functions $f$ on $\Bn$ with
$$\sup_{z\in \Bn} (1-|z|^2)^{-1/2} |\widetilde{\nabla }f(z)|<\infty.$$
\end{enumerate}
\end{coro}

\begin{proof}
Parts (a) and (b) are consequence of Proposition \ref{Ba}, the remarks preceding it, and Corollary \ref{C5}. Part (c) is just a
restatement of Proposition \ref{Ba}.
\end{proof}

According to \cite[Theorem 7.2]{ZhuBn}, the space $BMO^p_{-\frac12}$ contains non-constant functions. In fact, any function
in $\mathcal{B}^{\alpha}$ with $0<\alpha<1/2$ is in $BMO^p_{-\frac12}$. However, this space differs from $\mathcal{B}^{1/2}$.
See Corollary \ref{C5}.

In the next theorem we record some characterizations obtained for Bloch type spaces, which are of some independent interest.

\begin{theorem}\label{bmo}
Suppose $r>0$, $1\le p<\infty$, $\sigma>-1$, and $f\in H(\Bn)$. Let $\gamma\ge -1$ if $n=1$; and $\gamma>-1/2$ if $n>1$.
Then the following statements are equivalent.
\begin{itemize}
\item[(i)] $f\in \mathcal{B}^{1+\gamma}$.
\item[(ii)] There is a constant $C>0$ such that
$$\frac{(1-|z|^2)^{\gamma p}}{v_{\sigma}(D(z,r))}\int_{D(z,r)}|f(w)-f(z)|^p\,dv_{\sigma}(w)\le C$$
for all $z\in\Bn$.
\item[(iii)] There is a constant $C>0$ such that
$$\frac{(1-|z|^2)^{\gamma p}}{v_{\sigma}(D(z,r))}\int_{D(z,r)}|f(w)-\widehat{f_{r}}(z)|^p\,dv_{\sigma}(w)\le C$$
for all $z\in\Bn$.
\item[(iv)] For each $z\in\Bn$ there exists a complex number $\lambda_ z$ such that
$$\sup_{z\in\Bn}\frac{(1-|z|^2)^{\gamma p}}{v_{\sigma}(D(z,r))}\int_{D(z,r)}|f(w)-\lambda_ z|^p\,dv_{\sigma}(w)<\infty.$$
\item[(v)] For some (or all) $\eta >\max(-1,-1+\gamma p)$ and for each $c>\max(0,-2\gamma p)$ we have
$$\sup_{z\in\Bn}\int_{\Bn}|f(w)-f(z)|^p\,\frac{(1-|z|^2)^{c+\gamma p}}{|1-\langle z,w\rangle|^{n+1+c+\eta}}\,dv_{\eta}(w)<\infty.$$
\end{itemize}
\end{theorem}

\begin{proof}
In the proof of Proposition \ref{Ba}, we proved the implications (i)$\Rightarrow$(v)$\Rightarrow$(ii)$\Rightarrow$(i). The other
equivalences are obtained from Theorem \ref{BMO-S} and Lemma \ref{Lem-1}.
\end{proof}

When $\gamma=0$, the equivalences of (i)-(iv) in Theorem \ref{bmo} is just \cite[Theorem 5.22]{ZhuBn}, and the equivalence
with (v) appears in \cite{LW}.

\section{Bounded Hankel operators}

The main result of this section is the following result, which characterizes bounded Hankel operators induced by real-valued
symbols between weighted Bergman spaces.

\begin{theorem}\label{mt1}
Let $1<p\le q<\infty$, $\alpha>-1$, $\beta>-1$, $f\in L^q_{\beta}$, and
$$\gamma=\frac{n+1+\beta}{q}-\frac{n+1+\alpha}{p}.$$
Then $H_{f}^{\beta},H_{\bar{f}}^{\beta}:A^p_{\alpha}\rightarrow L^q_{\beta}$ are both bounded if and only if $f\in BMO^q_{\gamma}$.
\end{theorem}

We are going to prove Theorem \ref{mt1} with a series of lemmas and propositions. For $t\ge 0$ let
$$K_ z^{\beta,t}(w)=\frac{1}{(1-\langle w,z\rangle)^{n+1+\beta+t}}.$$
Also, for $f\in L^q_{\beta}$ and $z\in \Bn$, we consider the function $MO_{\beta,q,t} f$ defined by
$$MO_{\beta,q,t} f (z)=\big \|fh^t_z-  \overline{g_ z(z)} h^t _z \big \|_{q,\beta},$$
where, for $z\in \Bn$, the function $g_ z$ (that depends on $f$ and $t$) is given by
$$g_ z(w)=\frac{P_{\beta}(\bar{f} h^t_ z)(w)}{h^t_ z(w)},\qquad w\in \Bn,$$
with the function $h^t_ z$ defined by
$$h^t_z(w)=h_ z^{\alpha,\beta,t}(w)=\frac{K_z^{\beta,t}(w)}{\|K_z^{\beta,t}\|_{p,\a}},\qquad w\in \Bn.$$
Clearly, $\|h^t_z\|_{p,\a}=1$. When $t=0$, it is easily seen that $\overline{g_ z(z)}=B_{\beta} f(z)$ is the Berezin transform of $f$ at
the point $z$.  Also, since $h^t_ z(w)$ never vanishes on $\Bn$, the function $g_ z$ is holomorphic on $\Bn$.

\begin{lemma}\label{R-MO}
Let $1<p\le q<\infty$ and $\alpha,\beta>-1$. Let $\gamma$ be as in Theorem \ref{mt1} and $t\ge 0$ such that
$$n+1+\beta+t>(n+1+\a)/p.$$
If $MO_{\beta,q,t} f\in L^{\infty}(\Bn)$, then $f\in BMO^q_{\gamma}$.
\end{lemma}

\begin{proof}
Since $n+1+\beta+t>(n+1+\a)/p$, Lemma \ref{Ict} gives us the estimate
$$\|K_z^{\beta,t}\|_{p,\a}\asymp(1-|z|^2)^{(n+1+\a)/p-(n+1+\beta+t)}.$$
It follows that $\big [MO_{\beta,q,t} f (z)\big ]^q$ is comparable to
$$(1-|z|^2)^{\gamma q+(n+1+\beta)(q-1)+tq}\int_{\Bn} \frac{|f(w)-\overline{g_ z(z)}\,|^q}{|1-\langle z,w \rangle|^{(n+1+\beta+t)q}}
\,dv_{\beta}(w),$$
which dominates
$$\frac{(1-|z|^2)^{\gamma q}}{|D(z,r)|}\int_{D(z,r)}|f(w)-\overline{g_z(z)}|^q\,dv(w).$$
This shows that, if $MO_{\beta,q,t} f\in L^{\infty}(\Bn)$, the condition in Lemma \ref{Lem-1} is satisfied with
$\lambda_ z=\overline{g_ z(z)}$, so $f\in BMO^q_{\gamma}$.
\end{proof}


The following result gives the necessity in Theorem \ref{mt1}. It generalizes, in several directions, Proposition 8.19 in \cite{Zhu}, where the
method of proof is based on Hilbert space techniques. A different method was used in \cite{Zhu-Pac} to deal with the case $\alpha=\beta$
and $p=q$. Our method here is more flexible and allows us to obtain the result in much more generality.

\begin{proposition}\label{necessity}
Let $1<p,q<\infty$, $\alpha,\beta>-1$, and $t\ge 0$. Then for any $f\in L^q_{\beta}$ we have
$$
MO_{\beta,q,t}f(z)
\lesssim \big \|H^{\beta}_f h^t_z\big \|_{q,\beta}+ \big \|H^{\beta}_{\bar f} h^t_z\big \|_{q,\beta}.
$$
\end{proposition}

\begin{proof}
By the triangle inequality and the definition of Hankel operators, we have
\begin{eqnarray*}
MO_{\beta,q,t}f(z)&=&\big \|fh^t_ z-\overline{g_ z(z)}\, h^t_ z \big \|_{q,\beta}\\
&\le& \big \|fh^t_ z -P_{\beta}(f h^t_ z)\big \|_{q,\beta}+\big \|P_{\beta}(f h^t_ z)-\overline{g_ z(z)}\, h^t_ z\big \|_{q,\beta}\\
&=&\big \|H^{\beta}_ f h^t_ z\big \|_{q,\beta}+\big \|P_{\beta}(f h^t_ z)-\overline{g_ z(z)}\, h^t_ z\big \|_{q,\beta}.
\end{eqnarray*}
For any $g\in A^1_{\beta}(\Bn)$ it is easy to check that
\begin{equation}\label{Eq-H1}
\overline{g(z)} h^t_ z =P_{\beta+t}(\bar{g} h^t_ z).
\end{equation}
This together with the boundedness of $P_{\beta+t}$ on $L^q_{\beta}$ yields
\begin{eqnarray*}
\big\|P_{\beta}(fh^t_ z)-\overline{g_ z(z)}\,h^t_ z\big\|_{q,\beta}&=&\big\|P_{\beta}(fh^t_ z)-P_{\beta+t}(\overline{g_ z} \,h^t_ z)\big\|_{q,\beta}\\
&=&\big \|P_{\beta+t} \big (P_{\beta}(f h^t_ z)-\overline{g_ z} \,h^t_ z)\big )\big \|_{q,\beta}\\
&\le& \big \|P_{\beta+t}\big \|_{L^q_{\beta}} \cdot \big \|P_{\beta}(f h^t_ z)-\overline{g_ z} \,h^t_ z\big \|_{q,\beta}.
\end{eqnarray*}
Finally,
\begin{eqnarray*}
\big \|P_{\beta}(f h^t_ z)-\overline{g_ z} \,h^t_ z\big \|_{q,\beta}& \le &\big \|f h^t_ z -P_{\beta}(f h^t_ z)\big \|_{q,\beta}+
\big \|fh^t_ z -\overline{g_ z} \,h^t_ z\big \|_{q,\beta}\\
&=&\big \|H^{\beta}_ f h^t_ z\big \|_{q,\beta}+\big \|\bar{f}\,h^t_ z -g_ z \,h^t_ z\big \|_{q,\beta}\\
&=&\big \|H^{\beta}_ f h^t_ z\big \|_{q,\beta}+\big \|\bar{f}\,h^t_ z -P_{\beta}(\bar{f}\,h^t_ z)\big \|_{q,\beta}\\
&=&\big \|H^{\beta}_ f h^t_ z\big \|_{q,\beta}+\big \|H^{\beta}_{\bar{ f}} h^t_ z\big \|_{q,\beta}.
\end{eqnarray*}
This proves the result with constant $C=\big (1+\big \|P_{\beta+t}\big \|_{L^q_{\beta}}\big )$.
\end{proof}

Note that the proposition above does not require $p\le q$.
The next two propositions, which require the condition $p\le q$, will
establish the sufficiency of Theorem \ref{mt1}.

\begin{proposition}\label{L-BA}
Let $1<p\le q <\infty$, $\alpha>-1$, $\beta>-1$, and
$$\gamma=(n+1+\beta)/q-(n+1+\alpha)/p.$$
If  $f\in BA^q_{\gamma}$, then $H_ f^{\beta}:A^p_{\alpha}\rightarrow L^q_{\beta}$ is bounded.
\end{proposition}

\begin{proof}
Since $q>1$, the Bergman projection $P_{\beta}$ is bounded on $L^q_{\beta}$. Thus
$$
\big \|H^{\beta}_ f g \big \|_{q,\beta} \le \|fg\|_{q,\beta} +\|P_{\beta}(fg)\| _{q,\beta}
\lesssim \|fg\|_{q,\beta}=\|g\|_{L^q(d\mu_{f,\beta})}.
$$
The result then follows from Lemma \ref{CMf}.
\end{proof}

We note that the proof of the previous proposition also works for $1=p<q<\infty$. In order to show that $H^{\beta}_ f$ is bounded if
$f\in BO_{\gamma}$ with $\gamma<0$, we need the following result.

\begin{lemma}\label{LHb1}
Let $s\ge \beta>-1$, $1<q<\infty$, $f\in L^q_{\beta}$, and $g\in H^{\infty}$. Then
$$\big \|H_ f^{\beta} g \big \|_{q,\beta} \le C \,\big \|H_ f^{s} g \big \|_{q,\beta}.$$
\end{lemma}

\begin{proof}
Since $g\in H^{\infty}$, we have $gf\in L^q_{\beta}$. Also,
\begin{eqnarray*}
\big\| H_ f^{\beta} g \big \|_{q,\beta}&=& \big \|(I-P_{\beta})(gf) \big \|_{q,\beta}\\
&\le& \big \|(I-P_{s})(gf) \big \|_{q,\beta}+\big \|(P_ s-P_{\beta})(gf) \big \|_{q,\beta}\\
& =&\big \|H_ f^{s} g \big \|_{q,\beta}+\big \|(P_ {\beta}-P_{s})(gf) \big \|_{q,\beta}.
\end{eqnarray*}
Since $P_ s$ is bounded on $L^q_{\beta}$, the reproducing formula yields $P_{\beta} P_ s (gf)=P_ s(gf)$. Thus
$$(P_ {\beta}-P_{s})(gf)=(P_ {\beta}-P_{\beta} P_{s})(gf)=P_{\beta}(I-P_ s)(gf)=P_{\beta} (H_ f^s g).$$
This gives
$$\big \|(P_ {\beta}-P_{s})(gf) \big \|_{q,\beta} \le \big \|P_{\beta} \big \| \cdot \big \|H_ f^{s} g \big \|_{q,\beta},$$
so we obtain the desired inequality with $C=1+ \|P_{\beta} \|$.
\end{proof}

\begin{proposition}\label{L-BO}
Let $1<p\le q <\infty$, $\alpha>-1$, $\beta>-1$, and
$$\gamma=(n+1+\beta)/q-(n+1+\alpha)/p.$$
If $f\in BO_{\gamma}$, then $H_ f^{\beta}:A^p_{\alpha}\rightarrow L^q_{\beta}$ is bounded.
\end{proposition}

\begin{proof}
We first consider the case $\gamma \ge 0$. For  $g\in H^{\infty}$, which is dense in  $A^p_{\alpha}$, we have
\begin{eqnarray*}
\|H_ f ^{\beta} g\|^q_{q,\beta}& =&\int_{\Bn} |H_ f^{\beta} g(z)|^q\, dv_{\beta}(z)\\
&=&\int_{\Bn} \left |\int_{\Bn}\frac{(f(z)-f(w))\,g(w)}{(1-\langle z,w \rangle )^{n+1+\beta}} \,dv_{\beta}(w)\right |^q\, dv_{\beta}(z)\\
&\le& \int_{\Bn} \left (\int_{\Bn}\frac{|f(z)-f(w)|\,|g(w)|}{|1-\langle z,w \rangle |^{n+1+\beta}} \,dv_{\beta}(w)\right )^q\, dv_{\beta}(z).
\end{eqnarray*}
By Lemma \ref{L-CBO},
$$\|H_ f ^{\beta} g\|^q_{q,\beta}\lesssim \int_{\Bn} \left (\int_{\Bn}\frac{(1+\beta(z,w))\,|g(w)| \,dv_{\beta}(w)}{|1-\langle z,w
\rangle |^{n+1+\beta}\min (1-|z|,1-|w|)^{\gamma}} \right )^q\, dv_{\beta}(z).$$
Write
$$I_ 1(g)=\int_{\Bn} \left (\int_{\Bn}\frac{|g(w)|\,dv_{\beta}(w)}{|1-\langle z,w \rangle |^{n+1+\beta}\min (1-|z|,1-|w|)^{\gamma}}\right )^q
\, dv_{\beta}(z),$$
and split the inner integral in two parts, $I_{1,1}(g)$ over $|w|\le |z|$ and $I_{1,2}(g)$ over $|w|>|z|$. Since
$$\min (1-|z|,1-|w|)^{\gamma}= (1-|z|)^{\gamma}$$
for $|w|\le |z|$, we have
\begin{eqnarray*}
I_ {1,1}(g)&\lesssim&\int_{\Bn} \left (\int_{\Bn}\frac{|g(w)|\,dv_{\beta}(w)}{|1-\langle z,w \rangle |^{n+1+\beta}}\right )^q
\, dv_{\beta-q\gamma}(z)\\
&=&\int_{\Bn} \big |S_{b,c}(|g|)(z)\big |^q \, dv_{\beta-q\gamma}(z),
\end{eqnarray*}
where $S_{b,c}$ is the integral operator appearing in Theorem \ref{Z-T} with $b=\beta$ and $c=n+1+\beta$. Notice that
$\beta-q\gamma>-1$ is equivalent to
$$n\left(\frac1q-\frac1p\right)<\frac{1+\alpha}{p},$$
which is automatically satisfied since $p\le q$. Applying Theorem \ref{Z-T}, we obtain $I_{1,1}(g) \lesssim \|g\|_{p,\alpha}^q$,
provided $1+\a<p(1+b)$ and $c\le n+1+b+\lambda$. Since $b=\beta$ and
$$\lambda=\frac{n+1+(\beta-q\gamma)}{q}-\frac{n+1+\alpha}{p}=0,$$
the condition $c\le n+1+b+\lambda$ is satisfied with equality. It remains to check that the condition
\begin{equation}\label{Cab}
1+\a<p(1+b)=p(1+\beta)
\end{equation}
is satisfied. Since $\gamma \ge 0$ and $q\ge p$, we have
$$0\le \gamma =(n+1+\beta)/q-(n+1+\alpha)/p \le (\beta-\alpha)/p.$$
This gives $\alpha \le \beta$, so\eqref{Cab} holds since $p> 1$.

Similarly, we have
$$I_ {1,2}(g)\lesssim \int_{\Bn} \big |S_{b,c}(|g|)(z)\big |^q \, dv_{\beta}(z)$$
with $b=\beta-\gamma$, $c=n+1+\beta$, and $\lambda=\gamma$. We want to apply Theorem \ref{Z-T} to estimate $I_{1,2}(g)$.
In this case, the condition $c\le n+1+b+\lambda$ in Theorem~\ref{Z-T} holds with equality.
The other condition in Theorem \ref{Z-T} is $\alpha+1<p(1+\beta-\gamma)$, which is equivalent to
$$\frac{p(n+1+\beta)}{q}<p\beta+p+n.$$
If $q'$ is the conjugate exponent of $q$, the above condition is equivalent to
$$n\left(\frac1q-\frac1p\right)<\frac{1+\beta}{q'},$$
which is automatically satisfied since $p\le q$. Hence, by Theorem \ref{Z-T}, we have $I_{1,2}(g)\lesssim \|g\|^q_{p,\alpha}$. This
together with the previous estimate yields $I_ 1(g) \lesssim \|g\|_{p,\alpha}^q$. The remaining estimate
$I_ 2(g) \lesssim \|g\|_{p,\alpha}^q$ with
$$I_ 2(g):=\int_{\Bn} \left (\int_{\Bn}\frac{|g(w)|\,\beta(z,w)\,dv_{\beta}(w)}{|1-\langle z,w \rangle |^{n+1+\beta}
\min (1-|z|,1-|w|)^{\gamma}}\right )^q\, dv_{\beta}(z),$$
can be proved in a similar manner, using Proposition \ref{proj-ge1-beta} instead of Theorem \ref{Z-T}. The proof of the case
$\gamma\ge 0$ is now complete.

If $\gamma<0$ and $g\in H^{\infty}$, we use Lemma \ref{LHb1}, with $s\ge \beta$ big enough so that $p(s+\gamma+1)>\alpha+1$,
to obtain
$$\big \|H_ f ^{\beta} g \big \|^q_{q,\beta}\lesssim \big \|H_ f ^{s} g\big \|^q_{q,\beta}\le \int_{\Bn} \left[\int_{\Bn}\frac{|f(z)-f(w)|\,|g(w)|}
{|1-\langle z,w \rangle |^{n+1+s}} \,dv_{s}(w)\right]^q\, dv_{\beta}(z).$$
By Lemma \ref{L-CBO},
$$|f(z)-f(w)|\le C \frac{\beta(z,w)+1}{ (1-|w|)^{-\gamma}}\,|1-\langle z,w \rangle |^{-2\gamma}.$$
Therefore,
$$\big \|H_ f ^{\beta} g \big \|^q_{q,\beta}\lesssim \int_{\Bn} \left (\int_{\Bn}\frac{(\beta(z,w)+1)\,|g(w)|}{|1-\langle z,w
\rangle |^{n+1+s+2\gamma}} \,dv_{s+\gamma}(w)\right )^q\, dv_{\beta}(z),$$
and the boundedness of $H_f^\beta:A^p_\alpha\to L^q_\beta$ follows from Theorem \ref{Z-T} and
Proposition \ref{proj-ge1-beta} again.
\end{proof}

The proof of Theorem \ref{mt1} is now complete: the necessity of the condition $f\in BMO^q_{\gamma}$ follows from Lemma \ref{R-MO}
and Proposition \ref{necessity}. Since $f\in BMO^q_{\gamma}$ if and only if $\overline{f}\in BMO^q_{\gamma}$, the sufficiency is a
consequence of Theorem \ref{BMO-S}, Proposition \ref{L-BA}, and Proposition \ref{L-BO}.

As an immediate consequence of Theorem \ref{mt1} and Proposition \ref{Ba}, we obtain the following result that characterizes the
boundedness of Hankel operators with conjugate holomorphic symbols.

\begin{coro}
Let $f\in A^1_{\beta}$, $1<p\le q<\infty$, $\alpha>-1$, $\beta>-1$, and
$$\gamma=\frac{n+1+\beta}{q}-\frac{n+1+\alpha}{p}.$$
\begin{enumerate}
\item For $n=1$ we have
\begin{enumerate}
\item[(a)] If $\gamma \ge -1$, then $H_{\bar{f}}:A^p_{\alpha}\rightarrow L^q_{\beta}$ is bounded if and only if $f\in \mathcal{B}^{1+\gamma}$.
\item[(b)] If $\gamma < -1$, then $H_{\bar{f}}:A^p_{\alpha}\rightarrow L^q_{\beta}$ is bounded if and only if $f$ is constant.
\end{enumerate}
\item For $n>1$ we have
\begin{enumerate}
\item[(a)] If $\gamma > -1/2$, then $H_{\bar{f}}:A^p_{\alpha}\rightarrow L^q_{\beta}$ is bounded if and only if $f\in \mathcal{B}^{1+\gamma}$.
\item[(b)] If $\gamma < -1/2$, then $H_{\bar{f}}:A^p_{\alpha}\rightarrow L^q_{\beta}$ is bounded if and only if $f$ is constant.
\item[(c)] If $\gamma = -1/2$, then $H_{\bar{f}}:A^p_{\alpha}\rightarrow L^q_{\beta}$ is bounded if and only if
$$\sup_{z\in \Bn} (1-|z|^2)^{-1/2} |\widetilde{\nabla }f(z)|<\infty.$$
\end{enumerate}
\end{enumerate}
\end{coro}

\begin{proof}
Since $f\in A^1_{\beta}$, the Hankel operator $H_{\bar{f}}$ is densely defined. If $H_{\bar{f}}:A^p_{\alpha}\rightarrow L^q_{\beta}$ is bounded,
by testing the boundedness on the function 1 we see that $f\in A^q_{\beta}$. Since $\mathcal{B}^{1+\gamma}\subset A^q_{\beta}$, the
result follows from Theorem \ref{mt1} and Proposition \ref{Ba}.
\end{proof}

In the case $q=p=2$, this recovers the results of Janson and Wallst\'{e}n \cite{J,Wall},  where the case  $\gamma=-1$ for $n=1$ and the
case $\gamma=-1/2$ for $n>1$ were left open. Thus we have resolved these open cases.

\section{Weighted VMO spaces}

Let $\gamma\in\mathbb{R}$. For any positive radius $r$ and every exponent $p$ with $1\le p<\infty$, the space $VMO^p_{r,\gamma}$
consists of those  functions $f$ in $BMO^p_{r,\gamma}$ such that
$$\lim_{|z|\to1}(1-|z|^2)^{\gamma} MO_{p,r}(f)(z)=0.$$
Again, the space $VMO^p_{r,\gamma}$ is actually independent of the weight parameter $\sigma$. Similarly as before, for $r>0$, we
define $VO_{r,\gamma}$ as the space of functions $f$ in $BO_{r,\gamma}$ satisfying
$$\lim_{|z|\to 1} (1-|z|^2)^{\gamma}\omega_r(f)(z)=0,$$
and $VA^p_{r,\gamma}$ as the space of functions $f$ in $BA^p_{r,\gamma}$ satisfying
$$\lim_{|z|\to 1} (1-|z|^2)^{\gamma}\left[\widehat{|f|^p_ r}(z)\right]^{1/p}=0.$$
The following result shows that $VA^p_{r,\gamma}$ does not depend on $r$.

\begin{lemma}\label{VA}
Let $0<p\le q<\infty$, $\alpha>-1$, $\beta>-1$, $f\in L^q_{loc}(\Bn)$, and
$$\gamma=(n+1+\beta)/q-(n+1+\alpha)/p,\qquad d\mu_{f,\beta}=|f|^q dv_{\beta}.$$
The following are equivalent:
\begin{itemize}
\item[(i)] If $\{f_k\}$ is a bounded sequence in $A^p_{\alpha}$ and $f_k\to0$ uniformly on every compact subset of $\Bn$, then
$$\lim_{k\to\infty}\int_{\Bn}|f_k(z)|^q\,d\mu_{f,\beta}(z)=0.$$
\item [(ii)] $f\in VA^q_{r,\gamma}$ for some (or all) $r>0$.
\item[(iii)] The condition
$$\lim_{|z|\to1}(1-|z|^2)^{\gamma q} B_{c,\sigma}(|f|^q)(z)=0$$
holds for all $\sigma>\max(-1,-1+\gamma q)$ and all $c>\max(0,-\gamma q)$.
\end{itemize}
\end{lemma}

\begin{proof}
By the corresponding little-oh result of Theorem 50 in \cite{ZZ}, we know that (i) is equivalent to
$$\lim_{|z|\to1}\frac{\mu_{f,\beta}(D(z,r))}{(1-|z|^2)^{(n+1+\alpha)q/p}}=0$$
for some (or all) $r>0$.
The equivalence of (i) and (ii) is a consequence of this result and the fact that
$$\widehat{|f|^q_ r}(z)\asymp \frac{\mu_ {f,\beta} (D(z,r))}{(1-|z|^2)^{n+1+\beta}}.$$

That (iii) implies (ii) follows from the fact that
$$\widehat{|f|^q_ r}(z)\lesssim B_{c,\sigma}(|f|^q)(z),$$
which has been shown in the proof of Lemma~\ref{CMf}.

It remains to prove that (ii) implies (iii). Let $f\in VA^q_{r,\gamma}$. By definition, we have
$$(1-|z|^2)^{\gamma q} B_{c,\sigma}(\widehat{|f|^q_ r})(z)=(1-|z|^2)^{c+\gamma q}
\int_{\Bn}\frac{\widehat{|f|^q_ r}(w)\,dv_{\sigma}(w)}{|1-\langle z,w \rangle |^{n+1+c+\sigma}}.$$
For $0<s<1$ let
$$I_1(s)=(1-|z|^2)^{c+\gamma q}\int_{|w|\le s}\frac{\widehat{|f|^q_ r}(w)\,dv_{\sigma}(w)}{|1-\langle z,w \rangle |^{n+1+c+\sigma}},$$
and
$$I_2(s)=(1-|z|^2)^{c+\gamma q}\int_{s<|w|<1}\frac{\widehat{|f|^q_ r}(w)\,dv_{\sigma}(w)}{|1-\langle z,w \rangle |^{n+1+c+\sigma}}.$$
Let $\varepsilon$ be an arbitrary positive number. By (ii), and then by Lemma~\ref{Ict}, there exists an $s>0$ such that
$$I_2(s)\lesssim \,\varepsilon \, (1-|z|^2)^{c+\gamma q}\int_{s<|w|<1}\frac{(1-|w|^2)^{\sigma-\gamma q}
\,dv(w)}{|1-\langle z,w \rangle |^{n+1+c+\sigma}}\lesssim \, \varepsilon.$$
Since $f\in VA^q_{r,\gamma}\subset BA^q_{r,\gamma}$, we know that
$$\widehat{|f|^q_ r}(w)\lesssim (1-|w|^2)^{-q\gamma}.$$
Since $|1-\langle z,w \rangle| \gtrsim (1-|w|^2)$, we obtain
\begin{eqnarray*}
I_1(s)&\le& (1-|z|^2)^{c+\gamma q}\int_{|w|\le s}\frac{dv_\sigma(w)}{(1-|w|^2)^{n+1+c+\sigma+\gamma q}}\\
&\lesssim& \frac{(1-|z|^2)^{c+\gamma q}}{(1-s^2)^{n+1+c+\gamma q}}.
\end{eqnarray*}
Hence, we can find a $\delta\in(0,1)$ such that $I_1(s)\lesssim \,\varepsilon$ whenever $1-\delta<|z|<1$.
Combining the above two inequalities for $I_1(s)$ and $I_2(s)$ we deduce for $1-\delta<|z|<1$ that
$$(1-|z|^2)^{\gamma q} B_{c,\sigma}(\widehat{|f|^q_ r})(z)\lesssim \, \varepsilon.$$
Therefore,
$$\lim_{|z|\to1}(1-|z|^2)^{\gamma q} B_{c,\sigma}(\widehat{|f|^q_ r})(z)=0.$$
Let $d\mu_{f,\sigma}=|f|^q\,dv_{\sigma}$. Since
\begin{equation}\label{ave}
\widehat{|f|^q_r}(z)\asymp\widehat{\mu_{f,\sigma}}(z):=\frac{\mu_{f,\sigma}(D(z,r))}{(1-|z|^2)^{n+1+\sigma}},
\end{equation}
the above equation is equivalent to
$$\lim_{|z|\to1}(1-|z|^2)^{\gamma q} B_{c,\sigma}(\widehat{\mu_{f,\sigma}})(z)=0.$$
By \cite[Lemma 52]{ZZ}, we have
$$B_{c,\sigma}(\mu_{f,\sigma})(z)\lesssim B_{c,\sigma}(\widehat{\mu_{f,\sigma}})(z),$$
where
$$B_{c,\sigma}(\mu_{f,\sigma})(z)=(1-|z|^2)^{c}\int_{\Bn}\frac{d\mu_{f,\sigma}(w)}{|1-\langle z,w \rangle |^{n+1+c+\sigma}}.$$
Thus we obtain
$$\lim_{|z|\to1}(1-|z|^2)^{\gamma q} B_{c,\sigma}(\mu_{f,\sigma})(z)=0,$$
which is the same as
$$\lim_{|z|\to1}(1-|z|^2)^{\gamma q} B_{c,\sigma}(|f|^q)(z)=0.$$
This proves (ii) implies (iii) and completes the proof of the lemma.
\end{proof}

The next result shows that $VO_{r,\gamma}$ does not depend on $r$.

\begin{lemma}\label{VO}
Let $\gamma \in \mathbb{R}$ and $r_1, r_2>0$. If $f\in VO_{r_1,\gamma}$, then $f\in VO_{r_2,\gamma}$.
\end{lemma}

\begin{proof}
If $r_1>r_2$, the result is obvious. So we assume that $r_1<r_2$ and fix $z\in\Bn$. It follows from the continuity of $f$ on $\Bn$ that
$$\omega_{r_2}(f)(z)=\sup\{|f(z)-f(\zeta)|,\, \zeta\in\overline{D(z,r_2)}\},$$
and we can find $w\in\overline{D(z,r_2)}$ such that
$$|f(z)-f(w)|=\omega_{r_2}(f)(z).$$
Let $\lambda=\lambda(t)$, $0\le t\le 1$, be the geodesic in the Bergman metric from $z$ to $w$.
Then $\lambda$ lies entirely in $\overline{D(z,r_2)}$. As in the proof of Lemma~\ref{L-CBO},
we let $N=[r_2/r_1]+2$ and $t_i=i/N$, $0\le i\le N$, where $[x]$ denotes the largest integer less than or equal to $x$.
Set $z_i=\lambda(t_i)$, $0\le i\le N$. Since $N\ge r_2/r_1+1>r_2/r_1$, we have
$$\beta(z_{i-1},z_{i})=\frac{\beta(z,w)}{N}\le \frac{r_2}{N}<r_1.$$
Because $z_i$ is in the closure of $D(z,r_2)$, there exists a constant $K>0$, independent of $i$, such that
$$\frac1{K}(1-|z|^2)^{\gamma}\le (1-|z_i|^2)^{\gamma}\le K(1-|z|^2)^{\gamma}.$$
Since $f\in VO_{r_1,\gamma}$, we know that
$$\lim_{|z_i|\to1}(1-|z_i|^2)^{\gamma}\omega_{r_1}(f)(z_i)=0.$$
But $|z_i|\to1$ as $|z|\to1$. So for any $\eps>0$ there exists $\delta>0$ such that
$$(1-|z_i|^2)^{\gamma}\omega_{r_1}(f)(z_i)<\frac{\eps}{NK}$$
whenever $1-|z|<\delta$. Thus
\begin{eqnarray*}
|f(z)-f(w)|&\le&\sum_{i=1}^{N}|f(z_{i-1})-f(z_{i})|\le\sum_{i=1}^{N}\omega_{r_1}(f)(z_i)\\
&\le& \frac{\eps}{K(1-|z_i|^2)^{\gamma}}\le\frac{\eps}{K}\cdot\frac{K}{(1-|z|^2)^{\gamma}}\\
&=&\frac{\eps}{(1-|z|^2)^{\gamma}}.
\end{eqnarray*}
Therefore,
$$\omega_{r_2}(f)(z)=|f(z)-f(w)|\le\frac{\eps}{(1-|z|^2)^{\gamma}}$$
for $1-\delta<|z|<1$, which shows that
$$\lim_{|z|\to1}(1-|z|^2)^{\gamma}\omega_{r_2}(f)(z)=0,$$
or $f\in VO_{r_2,\gamma}$.
\end{proof}

Because of Lemmas~\ref{VA} and \ref{VO}, we can denote $VA^p_{r,\gamma}$ and $VO_{r,\gamma}$ by $VA^p_{\gamma}$
and $VO_{\gamma}$, respectively. Just as in the big-oh case, we have the following result for $VMO^p_{r,\gamma}$.

\begin{theorem}\label{VMO-S}
Suppose $r>0$, $\gamma \in \mathbb{R}$, $1\le p<\infty$, and $f\in BMO^p_{\gamma}$.
The following conditions are equivalent:
\begin{itemize}
\item[$(a)$] $f\in VMO^p_{r,\gamma}$.
\item[$(b)$] $f=f_ 1+f_ 2$ with $f_1\in VO_{\gamma}$ and $f_ 2 \in VA^p_{\gamma}$.
\item[$(c)$] For some (or all) $\sigma >\max(-1,-1+\gamma p)$ and for each $c\ge\max(n+1+\sigma, n+1+\sigma-2\gamma)$,
we have
$$\lim_{|z|\to1} \int_{\Bn} |f(w)-\widehat{f_ r}(z)|^p\frac{(1-|z|^2)^{c+\gamma p}}{|1-\langle z,w \rangle |^{n+1+c+\sigma}}
\,dv_{\sigma}(w)=0.$$
\item[$(d)$] For some ( or all) $\sigma >\max(-1,-1+\gamma p)$ and for each $c\ge\max(n+1+\sigma, n+1+\sigma-2\gamma)$,
there is a function $\lambda_ z$ such that
$$\lim_{|z|\to1}  \int_{\Bn} |f(w)-\lambda_ z|^p\frac{(1-|z|^2)^{c+\gamma p}}{|1-\langle z,w \rangle |^{n+1+c+\sigma}} dv_{\sigma}(w)=0.$$
\item[$(e)$] For some (or all) $\sigma >-1$ there is a function $\lambda_ z$ such that
$$\lim_{|z|\to1}\frac{(1-|z|^2)^{\gamma p}}{v_{\sigma}(D(z,r))}\int_{D(z,r)} |f(w)-\lambda_ z|^p\,dv_{\sigma}(w)=0.$$
\end{itemize}
\end{theorem}

\begin{proof}
That (c) implies (d) is obvious. That (d) implies (e) follows from the simple inequality
\begin{eqnarray*}
&&\frac{(1-|z|^2)^{\gamma p}}{v_{\sigma}(D(z,r))}\int_{D(z,r)} \!\! |f(w)-\lambda_ z|^p\,dv_{\sigma}(w)\\
&\lesssim&\int_{\Bn} \! |f(w)-\lambda_ z|^p\frac{(1-|z|^2)^{c+\gamma p}}{|1-\langle z,w \rangle |^{n+1+c+\sigma}} dv_{\sigma}(w).
\end{eqnarray*}
An easy modification of the proof of Lemma~\ref{Lem-1} shows that (e) implies (a). That (a) implies (b) follows easily from the
proof of (a) implying (b) in Theorem~\ref{BMO-S}.

Thus we only need to prove that (b) implies (c). Suppose that (b) holds. As in the proof of Theorem \ref{BMO-S}, from
Lemma \ref{VA} it is not difficult to see that (c) is satisfied for $f\in VA^p_{\gamma}$. Now, for $f\in VO_{\gamma}$, it is obvious
that $f\in BO_{\gamma}$. Set
$$I(z)=\int_{\Bn} |f(w)-\widehat{f_ r}(z)|^p\frac{(1-|z|^2)^{c+\gamma p}}{|1-\langle z,w \rangle |^{n+1+c+\sigma}} dv_{\sigma}(w).$$
Making the change of variables $w=\varphi_z(\zeta)$, we obtain
\begin{equation}\label{EqIZ0}
I(z)=\int_{\Bn}|f\circ\varphi_z(\zeta)-\widehat{f_ r}(z)|^p\frac{(1-|z|^2)^{\gamma p}}{|1-\langle z,\zeta \rangle |^{n+1+\sigma-c}}
\,dv_{\sigma}(\zeta).
\end{equation}
In the case $\gamma\ge 0$, it follows from the proof of Theorem~\ref{BMO-S} and the invariance of the Bergman metric that
\begin{eqnarray*}
|f\circ\varphi_z(\zeta)-\widehat{f_ r}(z)|&\lesssim& \|f\|_{BO_{\gamma}} \,\frac{\beta(\varphi_z(\zeta), z)+1}{\min (1-|z|,
1-|\varphi_z(\zeta)|)^{\gamma}}\\
&\lesssim& \|f\|_{BO_{\gamma}} \,\frac{\beta(\zeta,0)+1}{\min (1-|z|,1-|\varphi_z(\zeta)|)^{\gamma}}.
\end{eqnarray*}
Let
$$
E=\{\zeta\in\Bn:\, |\varphi_z(\zeta)|\le|z|\}.
$$
For $\zeta\in E$ we have $1-|\varphi_z(\zeta)|^2\ge 1-|z|^2$ and
\begin{equation}\label{I1}
|f\circ\varphi_z(\zeta)-\widehat{f_ r}(z)|^p(1-|z|^2)^{\gamma p}\lesssim (\beta(\zeta,0)+1)^p.
\end{equation}
For $\zeta\in\Bn\setminus E$ we have $1-|\varphi_z(\zeta)|^2\le 1-|z|^2$ and
\begin{eqnarray*}
|f\circ\varphi_z(\zeta)-\widehat{f_ r}(z)|^p(1-|z|^2)^{\gamma p}
&\lesssim& \frac{(\beta(\zeta,0)+1)^p(1-|z|^2)^{\gamma p}}{(1-|\varphi_z(\zeta)|^2)^{\gamma p}}\\
&\lesssim&  (\beta(\zeta,0)+1)^p(1-|\zeta|^2)^{-\gamma p}.
\end{eqnarray*}
Since $\sigma>-1+\gamma\, p\ge -1$, we have
$$\int_{\Bn}(\beta(\zeta,0)+1)^p\,dv_{\sigma}(\zeta)<\infty,$$
and
$$\int_{\Bn}(\beta(\zeta,0)+1)^p(1-|\zeta|^2)^{-\gamma p}\,dv_{\sigma}(\zeta)<\infty.$$
Let
$$H(\zeta)=\begin{cases}(\beta(\zeta,0)+1)^p,& \zeta\in E\\
(\beta(\zeta,0)+1)^p(1-|\zeta|^2)^{-\gamma p}),& \zeta\in \Bn\setminus E
\end{cases}$$
The above argument shows that $H(\zeta)$ is  in $L^1(\Bn,dv_{\sigma})$ and, since $c\ge n+1+\sigma$, we have
\begin{equation}\label{EqIZ1}
|f\circ\varphi_z(\zeta)-\widehat{f_ r}(z)|^p(1-|z|^2)^{\gamma p}\,|1-\langle z,\zeta\rangle |^{c-(n+1+\sigma)}\lesssim
H(\zeta)
\end{equation}
for $\gamma\ge0$ and all $z\in \Bn$.

If $\gamma <0$, it follows from the proof of Theorem~\ref{BMO-S} again that
\begin{eqnarray*}
|f\circ\varphi_z(\zeta)-\widehat{f_ r}(z)|&\lesssim& \|f\|_{BO_{\gamma}} \,\frac{\big (\beta(\varphi_z(\zeta), z)+1\big )}
{ (1-|z|^2)^{-\gamma}}\big | 1-\langle z,\varphi_ z(\zeta)\rangle |^{-2\gamma}\\
&=& \|f\|_{BO_{\gamma}} \,\frac{(\beta(\zeta,0)+1)(1-|\zeta|^2)^{\gamma}}{(1-|\varphi_ z(\zeta)|^2)^{\gamma}}\\
&=&\|f\|_{BO_{\gamma}} \,\frac{(\beta(\zeta,0)+1)\,|1-\langle z,\zeta\rangle|^{2\gamma}}{(1-|z|^2)^{\gamma}}.
\end{eqnarray*}
Since $c\ge n+1+\sigma-2\gamma p$, we also have
\begin{equation}\label{EqIZ2}
|f\circ\varphi_z(\zeta)-\widehat{f_ r}(z)|^p(1-|z|^2)^{\gamma p}\,|1-\langle z,\zeta\rangle |^{c-(n+1+\sigma)}\lesssim G(\zeta)
\end{equation}
for all $z\in \Bn$, where $G(\zeta)=(\beta(\zeta,0)+1)^p$ is in $L^1(\Bn,dv_{\alpha})$.

Fix any $\zeta\in \Bn$ and let $t=\beta(\zeta,0)$. Since $\beta(\varphi_z(\zeta),z)=\beta(\zeta,0)=t$ and $f\in VO_{\gamma}$,
we get
\begin{equation}\label{L0-Eq1}
\lim_{|z|\to1}|f\circ\varphi_z(\zeta)-f(z)|^p(1-|z|^2)^{\gamma p}\le \lim_{|z|\to1}(1-|z|^2)^{\gamma p}\omega_t(f)(z)^p=0.
\end{equation}
On the other hand, we have
\begin{eqnarray*}
|f\circ\varphi_z(\zeta)&-&\widehat{f_ r}(z)|^p\lesssim |f\circ\varphi_z(\zeta)-f(z)|^p+|f(z)-\widehat{f_ r}(z)|^p \\
&\lesssim& |f\circ\varphi_z(\zeta)-f(z)|^p+\\
&&+\ \frac1{v_{\sigma}(D(z,r))^p}\int_{D(z,r)}|f(z)-f(t)|^p\,dv_{\sigma}(t)\\
&\lesssim& |f\circ\varphi_z(\zeta)-f(z)|^p+\omega_r(f)(z)^p.
\end{eqnarray*}
Therefore,
\begin{eqnarray*}
&&|f\circ\varphi_z(\zeta)-\widehat{f_ r}(z)|^p(1-|z|^2)^{\gamma p}\\
&\lesssim& |f\circ\varphi_z(\zeta)-f(z)|^p(1-|z|^2)^{\gamma p}+\omega_r(f)(z)^p(1-|z|^2)^{\gamma p},
\end{eqnarray*}
which tends to zero as $|z|\to 1$, because $f\in VO_{\gamma}$ and \eqref{L0-Eq1}. Since $c\ge n+1+\sigma$, we also have
$$\lim_{|z|\rightarrow 1^{-}}|f\circ\varphi_z(\zeta)-\widehat{f_ r}(z)|^p(1-|z|^2)^{\gamma p}\,|1-\langle z,
\zeta\rangle |^{c-(n+1+\sigma)}=0.$$
Thus in all cases, due to \eqref{EqIZ1} and \eqref{EqIZ2}, we can apply Lebesgue's Dominated Convergence Theorem (bearing
in mind the expression for $I(z)$ given in \eqref{EqIZ0}) to obtain $I(z)\to0$ as $|z|\to1^-$, which is (c). This completes the
proof of the theorem.
\end{proof}

Since condition (b) in the theorem above is independent of $r$, we see that the space $VMO^p_{\gamma,r}$ is actually
independent of $r$. Thus we will simply use the notation $VMO^p_\gamma$.

Notice that in (c) and (d) of Theorem~\ref{VMO-S} we require a somehow stronger condition $c\ge n+1+\sigma$ than $c>0$ in
Theorem~\ref{BMO-S}. It would be nice to know whether it is possible to replace condition $c\ge n+1+\sigma$ by $c>0$ with
$c>-\gamma p$ in (c) and (d) here. Anyway, for our main purpose here (to characterize compactness of Hankel operators)
condition $c\ge n+1+\sigma$ is enough.

For $\alpha>0$ let $\mathcal{B}^{\alpha}_ 0=\mathcal{B}^{\alpha}_ 0(\Bn)$ denote the closure of the set of polynomials in
$\mathcal{B}^{\alpha}$. The space $\mathcal{B}^{\alpha}_ 0$ consists exactly of those holomorphic functions $f$ such that
$$\lim_{|z|\to1^-}(1-|z|^2)^{\alpha} \,|\nabla f(z)|=0.$$
As before, the complex gradient can be replaced by the radial derivative $Rf$. Furthermore, for $\alpha>1/2$, a function $f$ is in
$\mathcal{B}^{\alpha}_ 0$ if and only if the function $(1-|z|^2)^{\alpha-1} \,|\widetilde{\nabla} f(z)|$ belongs to ${\mathbb C}_ 0(\Bn)$.
Again, we refer to \cite[Chapter 7]{ZhuBn} for all these facts. With minor modifications in the proof of Proposition \ref{Ba} together
with Corollary \ref{C5} we obtain the following result.

\begin{proposition}\label{Ba-0}
Let $\gamma \in \mathbb{R}$ and $1\le p<\infty$. Then $H(\Bn)\cap VMO^p_{\gamma}=\mathcal{B}^{1+\gamma}_{0}$
for $n=1$ and $\gamma >-1$, or for $n>1$ and $\gamma>-1/2$. In all other cases, the space $H(\Bn)\cap VMO^p_\gamma$
consists of only constants.
\end{proposition}

\begin{proof}
The details are left to the interested reader.
\end{proof}

\section{Compact Hankel operators}

In this section we prove the following characterization of compact Hankel operators between weighted Bergman spaces.

\begin{theorem}\label{compact}
Let $1<p\le q<\infty$, $\alpha,\beta>-1$, $f\in L^q_{\beta}$, and
$$\gamma=\frac{n+1+\beta}{q}-\frac{n+1+\alpha}{p}.$$
Then both $H^{\beta}_f$, $H^{\beta}_{\overline{f}}:\,A^p_{\alpha}\rightarrow L^q_{\beta}$ are compact if and only if
$f\in VMO^q_{\gamma}$.
\end{theorem}

Again, we are going to prove Theorem~\ref{compact} with several lemmas. We begin with the necessity.

\begin{lemma}\label{comp-n}
Let $p,q,\alpha,\beta$ and $\gamma$ be as in Theorem~\ref{compact}. If both $H^{\beta}_f$, $H^{\beta}_{\overline{f}}:\,A^p_{\alpha}
\rightarrow L^q_{\beta}$ are compact, then $f\in VMO^q_{\gamma}$.
\end{lemma}

\begin{proof}
Fix a nonnegative $t$ such that $n+1+\beta+t>(n+1+\a)/p$. It is easy to see that $h_ z^t$ converges to zero uniformly on compact subsets of
$\Bn$ as $|z|\rightarrow 1^{-}$. Since each $h_z^t$ is a unit vector in $A^p_\alpha$, we conclude that $h_z^t\to0$ weakly in
$A^p_\alpha$ as $|z|\to1^-$. It follows from the compactness of $H^\beta_f$ that
$$\lim_{|z|\rightarrow 1^{-}}\|H_ f^{\beta}h^t_ z\|_{q,\beta}=0.$$
The same is true if $f$ is replaced by $\overline{f}$. By Proposition~\ref{necessity} we have
$$\lim_{|z|\rightarrow 1^{-}} MO_{q,\beta,t}(f)(z)=0.$$
In other words, we have
$$\lim_{|z|\rightarrow 1^{-}} (1-|z|^2)^{c+\gamma q}\int_{\Bn} \big|f(w)-\overline{g_ z(z)}\,\big |^q \,\frac{dv_{\beta}(w)}
{|1-\langle z,w \rangle |^{n+1+c+\beta}}=0,$$
where $c=(q-1)(n+1+\beta)+tq$. This implies condition (e) in Theorem 5.3 with $\lambda_ z=\overline{g_ z(z)}$,
so $f\in VMO^q_{\gamma}$.
\end{proof}

The sufficiency will follow from the next two results.

\begin{lemma}\label{CVA}
Let $p,q,\alpha,\beta$ and $\gamma$ be as in Theorem~\ref{compact}.
If $f\in VA^q_{\gamma}$, then $H^{\beta}_f:\,A^p_{\alpha}\rightarrow L^q_{\beta}$ is compact.
\end{lemma}

\begin{proof}
Let $\{g_n\}$ be a bounded sequence in $A^p_{\alpha}$ converging to zero uniformly on compact subsets of $\Bn$. We must
prove that $\|H^{\beta}_f g_n\|_{q,\beta} \rightarrow 0$. Following the proof of Proposition~\ref{L-BA}, we know that
$$\big \|H^{\beta}_f g_n\big \|_{q,\beta} \lesssim \|g_n\|_{L^q(d\mu_{f,\beta})}$$
with $d\mu_{f,\beta}=|f|^q dv_{\beta}$. The desired result then follows from Lemma \ref{VA}.
\end{proof}

\begin{lemma}\label{CVO}
Let $p,q,\alpha,\beta$ and $\gamma$ be as in Theorem~\ref{compact}. If $f\in VO_{\gamma}$,
then $H^{\beta}_f:\,A^p_{\alpha}\rightarrow L^q_{\beta}$ is compact.
\end{lemma}

\begin{proof}
Let $\{g_n\}$ be a bounded sequence in $A^p_{\alpha}$ converging to zero uniformly on compact subsets of $\Bn$.
By Lemma \ref{LHb1} and the density of $H^{\infty}$ in $A^p_{\alpha}$, for any $\delta\ge\beta$
we have
$$\big \|H^{\beta}_fg \big \|_{q,\beta}\le C \,\big \|H^{\delta}_fg \big \|_{q,\beta},\qquad g\in A^p_{\alpha}.$$
We will be done if we can prove that
$$\lim_{n\to\infty} \big \|H^{\delta}_f g_n \big \|_{q,\beta}=0$$
for some $\delta\ge \beta$.

Since $\beta>-1$, we can find some $\eta>0$ satisfying $\beta-\eta \max(q,q')>-1$. We then choose some $\delta\ge \beta$ large
enough so that $c=\eta q +\delta-\beta$ satisfies
$$c\ge n+1+\sigma+\max(0,-2\gamma q),$$
with $\sigma =\beta-\eta q$. In fact, this is the same as
$$\delta \ge n+1+2\beta+\max(0,-2\gamma q) -2\eta q.$$
So the choice $\delta=n+1+2\beta+\max(0,-2\gamma q)$ works. Let $\varepsilon>0$ be arbitrary. Since $VO_{\gamma}
\subset VMO^q_{\gamma}$, by part (c) of Theorem~\ref{VMO-S} and with the above $c$ and $\sigma$,
we may choose $t_ 1$ sufficiently close to 1 so that
\begin{equation}\label{eps}
(1-|w|^2)^{(\eta q+\delta-\beta)+\gamma q}\int_{\Bn} \frac{|f(z)-\widehat{f}_ r(w)|^q}
{|1-\langle z,w \rangle |^{n+1+\delta}} \,dv_{\beta-\eta q}(z)<\eps
\end{equation}
for all $t_1<|w|<1$.

Fix $r>0$. By the definition of $ VO_{\gamma}$, there exists $t_ 2$, $0<t_ 2<1$, such that
$$\omega_{r}(f)(w)<\eps (1-|w|^2)^{-\gamma},\qquad |w|>t_ 2.$$
We have
$$\|H^{\delta}_f g_n\|^q_{q,\beta}\le \int_{\Bn} \left (\int_{\Bn}\frac{|f(z)-f(w)|\,|g_ n(w)|}
{|1-\langle z,w \rangle |^{n+1+\delta}}\,dv_{\delta}(w)\right )^q\, dv_{\beta}(z).$$
Let $t=\max(t_ 1,t_ 2)$ and split the inner integral above in two parts: one for $|w|\le t$ and the other for $|w|>t$.
The integral on $|w|\le t$ can be made as small as we want because of the uniform convergence to zero on compact subsets
of $g_ n$. For the other, we will use our assumption $f\in VO_{\gamma}$.

Since
$$|f(z)-f(w)|\le |f(z)-\widehat{f}_r(w)|+|f(w)-\widehat{f}_r(w)|,$$
we get two integrals. The first one involves the function
$$I_1(z)=\int_{|w|>t} \frac{|f(w)-\widehat{f}_r(w)|\,|g_n(w)|}{|1-\langle z,w \rangle |^{n+1+\delta}} \,dv_{\delta}(w).$$
Since
$$f(w)-\widehat{f}_r(w)=\frac{1}{v_{\delta}(D(w,r))}\int_{D(w,r)}\big (f(w)-f(\zeta)\big )\,dv_{\delta}(\zeta),$$
we obtain for $|w|>t$ that
$$|f(w)-\widehat{f}_r(w)|\le \omega_{r}(f)(w)< \varepsilon \,(1-|w|^2)^{-\gamma}.$$
Therefore,
\begin{eqnarray*}
I_ 1&:=&\int_{\Bn} I_ 1(z)^q \,dv_{\beta}(z)\lesssim \varepsilon^q\, \int_{\Bn} \left(\int_{\Bn}
\frac{|g_n(w)| \,dv_{\delta-\gamma}(w)}{|1-\langle z,w \rangle |^{n+1+\delta}} \right )^q  \,dv_{\beta}(z)\\
&=& \varepsilon^q \int_{\Bn} S_{b,d} (|g_ n|)(z)^q \,dv_{\beta}(z),
\end{eqnarray*}
with $d=n+1+\delta$ and $b=\delta-\gamma$. Now we want to apply Theorem~\ref{Z-T} to show that
$S_{b,d}:L^p_{\alpha}\rightarrow L^q_{\beta}$ is bounded. In the notation of Theorem~\ref{Z-T} we have $\lambda=\gamma$
and
$$n+1+b+\lambda=n+1+\delta=d.$$
It remains to check the condition
$$\alpha+1<p(b+1)=p(1+\delta-\gamma),$$
which is easily seen to be equivalent to
\begin{equation}\label{delta}
n\left(\frac1q-\frac1p\right)<(1+\delta)-\frac{1+\beta}{q},
\end{equation}
By the proof of Proposition~\ref{L-BO}, we have
$$n\left(\frac1q-\frac1p\right)<\frac{1+\beta}{q'}=(1+\beta)-\frac{1+\beta}{q},$$
where $q'$ is the conjugate exponent of $q$. Since $\beta\le \delta$, we see that (\ref{delta}) is indeed true.
Therefore, by Theorem~\ref{Z-T}, we have
$$I_ 1 \lesssim \varepsilon ^q\,\|g_ n\|^q_{p,\alpha} \le C \,\varepsilon^q.$$

It remains to deal with
$$I_2:=\int_{\Bn} I_ 2(z)^q \,dv_{\beta}(z),$$
where
$$I_ 2(z)=\int_{|w|>t} \frac{|f(z)-\widehat{f}_ r(w)|\,\,|g_ n(w)|}{|1-\langle z,w \rangle |^{n+1+\delta}} \,dv_{\delta}(w).$$
By H\"{o}lder's inequality and Lemma~\ref{Ict},
\begin{eqnarray*}
I_ 2(z)^q&\lesssim& \left[\int_{|w|>t} \frac{|f(z)-\widehat{f}_ r(w)|^q\,\,|g_ n(w)|^q}
{|1-\langle z,w \rangle |^{n+1+\delta}} \,dv_{\delta+\eta q}(w)\right]\,\cdot\\
&&\ \cdot\,\left[\int_{\Bn} \frac{ dv_{\delta-\eta q'}(w)}{|1-\langle z,w \rangle |^{n+1+\delta}}  \right ]^{q/q'}\\
& \lesssim& (1-|z|^2)^{-\eta q} \left (\int_{|w|>t}
\frac{|f(z)-\widehat{f}_ r(w)|^q\,\,|g_ n(w)|^q}{|1-\langle z,w \rangle |^{n+1+\delta}} \,dv_{\delta+\eta q}(w)\right).
\end{eqnarray*}
Thus $I_ 2$ is dominated by
$$\int_{|w|>t} |g_ n(w)|^q \left[(1-|w|^2)^{\eta q}\int_{\Bn} \frac{|f(z)-\widehat{f}_ r(w)|^q}
{|1-\langle z,w \rangle |^{n+1+\delta}} \,dv_{\beta-\eta q}(z) \right]\,dv_{\delta}(w).$$
By \eqref{eps}, we get
$$I_2\lesssim \varepsilon \int_{|w|>t} |g_n(w)|^q\, dv_{\beta-\gamma q}(w)\lesssim \varepsilon \,\|g_n\|^q_{A^p_{\alpha}}.$$
For the last inequality we used the fact that $A^p_{\alpha}\subseteq A^{q}_{\beta-\gamma q}$, which can be obtained from
Theorem 69 in \cite{ZZ}. Putting everything together we conclude that $\|H^{\delta}_f g_n \big \|_{q,\beta}\rightarrow 0$ as
$n\to\infty$. This finishes the proof.
\end{proof}

To summarize, the necessity of Theorem \ref{compact} is proved by Lemma~\ref{comp-n}. Since $f\in VMO^q_{\gamma}$ if
and only if $\overline{f}\in VMO^q_{\gamma}$, the sufficiency is a consequence of Theorem \ref{VMO-S}, Lemma \ref{CVA},
and Lemma \ref{CVO}.

As a direct consequence of Theorem \ref{compact} and Proposition \ref{Ba-0}, we obtain the following characterization of compactness
of Hankel operators with conjugate holomorphic symbols.

\begin{coro}
Let $f\in A^1_{\beta}$, $1<p\le q<\infty$, $\alpha,\beta>-1$, and
$$\gamma=\frac{n+1+\beta}{q}-\frac{n+1+\alpha}{p}.$$
 If $n=1$ and $\gamma>-1$, or if $n>1$ and $\gamma>-1/2$, then $H^\beta_{\bar{f}}:A^p_{\alpha}\rightarrow L^q_{\beta}$
 is compact if and only if $f\in \mathcal{B}^{1+\gamma}_{0}$. In all other cases, $H^\beta_{\bar f}:A^p_\alpha\to L^q_\beta$
 is compact if and only if $f$ is constant.
\end{coro}


\end{document}